\documentclass[11pt]{amsart}
\usepackage{amssymb,enumerate,mathrsfs,geometry,tgpagella}
\usepackage[T1]{fontenc}
\usepackage{xcolor}
\usepackage[pagebackref=true, colorlinks, linkcolor=blue, citecolor=red]{hyperref}
\usepackage{microtype}

\newtheorem{theorem}{Theorem}[section]
\newtheorem{proposition}[theorem]{Proposition}
\newtheorem{corollary}[theorem]{Corollary}
\newtheorem{lemma}[theorem]{Lemma}
\newtheorem{maintheorem}{Theorem}

\theoremstyle{definition}

\newtheorem{example}[theorem]{Example}
\theoremstyle{plain}
\newtheorem*{singerthm}{Singer's Theorem}
\numberwithin{equation}{section}
\begin{document}
	\title[Liouvillian first integrals]
{Liouvillian first integrals of differential equations\raisebox{0.2ex}{:}\\ A Galoisian approach}

\author[P. Kumbhakar]{Partha Kumbhakar}
\address{Department of Mathematics, Ben-Gurion University of the Negev, Be'er Sheva 8410501, Israel}
\email{kumbhaka@post.bgu.ac.il}
\thanks{The first author was partially supported by the Israel Science Foundation
(ISF) grant no.\ 2958/24 through a postdoctoral fellowship at Ben-Gurion
University of the Negev. The second and third authors acknowledge the
support of the DST-FIST grant SR/FST/MS-I/2019/46(C)}

\author[C. Sahu]{Chitrarekha Sahu}
\address{Department of Mathematical Sciences, Indian Institute of Science Education and Research (IISER) Mohali, Sector 81, S.A.S. Nagar, Punjab 140306, India}
\email{ph20012@iisermohali.ac.in}
\author[V. R. Srinivasan]{Varadharaj R. Srinivasan}
\address{Department of Mathematical Sciences, Indian Institute of Science Education and Research (IISER) Mohali, Sector 81, S.A.S. Nagar, Punjab 140306, India}
\email{ravisri@iisermohali.ac.in}

	%\subjclass[2020]{Primary 12H05; Secondary 12H20, 34A05, 34M15}
	%\keywords{Liouvillian first integral, differential Galois theory, Picard-Vessiot ring, complete Picard-Vessiot closure, Singer's theorem, non-autonomous system}
	
	\begin{abstract}
Using differential Galois theory---specifically, Magid's theory of the
complete Picard-Vessiot closure---we show that a system of first-order
differential equations in $n$ variables admits a Liouvillian first integral
if and only if it admits a first integral in a Liouvillian Picard-Vessiot
extension, if and only if it admits a first integral in its Picard-Vessiot
ring. Consequently, a system with a Liouvillian first integral admits one in
a differential field obtained from the field of rational functions by first
taking a finite algebraic extension, then adjoining an exponential of an
integral, and then an integral. This yields a new proof of Singer's theorem
on Liouvillian first integrals of autonomous planar systems (Trans.\ Amer.\
Math.\ Soc.\ 333, 1992) and of its recent extension to autonomous systems in
$n$ variables by Aziz et al. (arXiv:2512.15522).
Our results hold over any differential field of characteristic zero with an
algebraically closed field of constants; in particular, they apply to non-autonomous systems. Finally, we exhibit a system that admits a first integral in a non-Liouvillian Picard-Vessiot extension but none in its
Picard-Vessiot ring and thus establishing that the
Liouvillian hypothesis cannot be dropped.
\end{abstract}
	\vspace*{-0.7cm} 
		\maketitle
        \vspace*{-0.7cm}
        {\small \tableofcontents}

	%%%%%%%%%%%%%%%%%%%%%%%%%%%%%%%%%%%%%%%%%%%%%%%%%%%%%%%%%%%%%%%%%%%%%%%%%%%%%
	\section{Introduction}\label{intro}
	%%%%%%%%%%%%%%%%%%%%%%%%%%%%%%%%%%%%%%%%%%%%%%%%%%%%%%%%%%%%%%%%%%%%%%%%%%%%%
Let $k$ be a differential field of characteristic zero with a derivation $'$, and assume its field of constants $C := \{a \in k : a' = 0\}$ is algebraically closed. Consider the system of first-order differential equations in $n$ variables
\begin{equation}\label{system-fode}\tag{S}
	\begin{cases}\begin{aligned}
		y_1' &= f_1(y_1, \dots, y_n),\\
		&\vdots \\
		y_n' &= f_n(y_1, \dots,y_n),
        \end{aligned}
	\end{cases}
    \end{equation}
where each $f_i$ is a rational function in $n$ variables with
coefficients in $k$. Let $K:=k(y_1,\dots,y_n)$ be the field of rational
functions, let $\delta_i:=\partial/\partial y_i$ for $i=1,\dots,n$, and let
$\delta_0$ be the unique extension of the derivation $'$ of $k$ to $K$ with
$\delta_0y_i=0$ for all $i$. These $n+1$ derivations commute, so $K$ is a
$\Delta$-field, that is, a field equipped with the set
$\Delta:=\{\delta_0,\delta_1,\dots,\delta_n\}$ of commuting derivations.
Its field of $\Delta$-constants, defined by
$C(K,\Delta):=\{a\in K : \delta a=0 \text{ for all } \delta\in\Delta\}$,
is $C$. We set
\begin{equation*}
	f_0:=1, \qquad
	D:=\sum_{i=0}^{n}f_i\delta_i=\delta_0+f_1\delta_1+\cdots+f_n\delta_n,
\end{equation*}
so that $D$ is the unique extension of $'$ to $K$ with $Dy_i=f_i$. The system \eqref{system-fode} is called \emph{autonomous} when $k=C$; then
$\delta_0$ is the zero derivation and is omitted, so that
$\Delta=\{\delta_1,\dots,\delta_n\}$,
$D=f_1\delta_1+\cdots+f_n\delta_n$, and all indices run from $1$ to $n$. When $k\neq C$, the system is called
\emph{non-autonomous}.

Let $R$ be a $\Delta$-ring extension of $K$, that is, $R$ is a ring extension of $K$ equipped with a set of commuting derivations extending those in $\Delta$ on $K$. Since $K \subset R$, the derivation $D$ extends to $R$ by the same formula. A \emph{first integral} of \eqref{system-fode} in $R$ is an element $e \in R$ such that $De = 0$ and $\delta_i e \neq 0$ for some $0 \le i \le n$. It is called a \emph{Liouvillian first integral} if it lies in a Liouvillian extension of $K$---a $\Delta$-field extension with no new constants, obtained by successively adjoining algebraic elements, integrals, or exponentials of integrals. For $k = \mathbb{C}$ with the zero derivation, the condition $De = 0$ provides the differential-algebraic analogue of the classical notion: an analytic
function $e(y_1,\dots,y_n)$ is constant along every solution of \eqref{system-fode} if and only if $De=0$ \cite[\S2]{Sin92}.

A classical problem is to determine which systems admit Liouvillian first integrals and what form these integrals can take. For autonomous planar polynomial systems, Singer \cite[Theorem 1 and Corollary]{Sin92} gave a complete and remarkably explicit answer.

\begin{singerthm}
The system
\begin{equation*}
	\begin{cases}
		y_1' = f_1(y_1,y_2)\\
		y_2' = f_2(y_1,y_2)
	\end{cases}
\end{equation*}
where $f_1,f_2\in\mathbb{C}[y_1,y_2]$, has a Liouvillian first integral
if and only if the differential form $f_2\,dy_1-f_1\,dy_2$ has an
integrating factor of the form
$\lambda=\exp\!\int\left(U\,dy_1+V\,dy_2\right)$ with $U,V\in\mathbb{C}(y_1,y_2)$
such that $\delta_2U=\delta_1V$, in which case
\begin{equation*}
	e=\int \lambda\,(f_2\,dy_1-f_1\,dy_2)
\end{equation*}
is a Liouvillian first integral.
\end{singerthm}
	
The significance of this result lies in the explicit form of the integral.
While a Liouvillian first integral may, a priori, lie in any Liouvillian extension, for planar systems there exists one in a much
simpler Liouvillian extension of $\mathbb{C}(y_1,y_2)$, obtained by adjoining an exponential of an integral followed by an
integral.

Singer's theorem has since been extended and reproved by several authors,
using a variety of techniques: Christopher \cite{Chr99}
recovered the Darbouxian shape of the integrating factor from invariant
algebraic curves; \.Zo\l\k{a}dek \cite{Zol98} gave a version for rational
$1$-forms in several variables by way of the extended monodromy group;
Moulin Ollagnier \cite{Oll96} treated homogeneous systems in three
variables; Casale \cite{Cas11} approached the problem through the Malgrange
Galois groupoid; and Zhang \cite{Zha16} related Liouvillian integrability of polynomial
systems in $n$ variables to the existence of Darboux Jacobian multipliers. Most recently, Aziz, Christopher, Pantazi and Walcher \cite{ACPW25} showed
that an autonomous rational system in $n>2$ variables with a Liouvillian
first integral has one obtained by two successive quadratures from a
finite algebraic extension of the rational function field. These works draw on tools such as monodromy, foliations, invariant
hypersurfaces, and Laurent and Puiseux series expansions.
	
	In this article, we use the Picard-Vessiot theory of $\Delta$-fields, developed by
Kolchin \cite{Kol52}, and the theory of complete Picard-Vessiot closures, developed by Magid \cite{Mag01, Mag02, Mag22}, to study Liouvillian first integrals of the (not necessarily autonomous) system \eqref{system-fode}. A \emph{Picard-Vessiot extension} of $K$ for an integrable system of
linear differential equations
\begin{equation}\label{lpde}
	\delta_iY=A_iY, \qquad A_i\in\mathrm{M}_m(K), \qquad 0\leq i\leq n,
\end{equation}
is a minimal $\Delta$-field extension $E$ of $K$ containing a fundamental solution matrix $Z\in\mathrm{GL}_m(E)$ of \eqref{lpde}. The $\Delta$-subring
$R:=K[Z,1/\det Z]\subseteq E$ is called the \emph{Picard-Vessiot
ring}. It is independent of the choice of $Z$, and $E$ is its field of
fractions. The \emph{differential Galois group} is the group of
$K$-$\Delta$-automorphisms of $E$ and it naturally admits a structure of a linear algebraic group
over $C$. A \emph{Liouvillian Picard-Vessiot extension} is a $\Delta$-field
extension of $K$ which is both a Liouvillian extension and a
Picard-Vessiot extension for some integrable system. Such extensions
are characterised by their differential Galois groups: a Picard-Vessiot extension is Liouvillian if and only if the
identity component of its differential Galois group is solvable \cite[Theorem~4]{Kol52}. Our proofs rest on the structure theory of
these solvable groups. The starting point of our investigation is that, when searching for first integrals in a Liouvillian Picard-Vessiot extension, it suffices to look in its Picard-Vessiot ring.

\begin{maintheorem}\label{main-thm-A}
    If the system \eqref{system-fode} admits a first integral in a Liouvillian
    Picard-Vessiot extension of $K$, then it admits a first
    integral in its Picard-Vessiot ring.
\end{maintheorem}
	
Any Liouvillian extension can be embedded in a tower of Liouvillian Picard-Vessiot extensions. Consequently, every Liouvillian first integral of \eqref{system-fode} must belong to such a tower.  Using Theorem~\ref{main-thm-A} and Magid's theory of the complete Picard-Vessiot closure, together with the fact that, for a first integral $e$, the vector $(\delta_0 e,\delta_1 e,\dots,\delta_n e)^{\mathsf t}$ satisfies the linear differential equation
\begin{equation*}
	DY=-(\delta_i f_j)_{0\leq i,j\leq n}\,Y
\end{equation*}
over the differential field $K$ with derivation $D$, we show that a system admitting a Liouvillian first integral must admit one in a single Liouvillian Picard-Vessiot extension. Note that $(\delta_i f_j)$ is the transpose of the \emph{Jacobian matrix} associated to the rational functions $f_0, f_1,\dots,f_n$. Applying Theorem~\ref{main-thm-A} once more, together with the Lie-Kolchin theorem on solvable linear algebraic groups, we then show that one such first integral has the form described in the next theorem.

\begin{maintheorem}\label{main-thm-B}
If the system \eqref{system-fode} admits a Liouvillian first integral, then
there exist a Liouvillian first integral $e$ and a finite algebraic
extension $\widetilde{K}$ of $K$ such that, for all $0\leq i,j\leq n$ with
$\delta_j e\neq 0$,
\begin{enumerate}[(i)]
	\item \label{thmB1} $\dfrac{\delta_i e}{\delta_j e}\in\widetilde{K},$
	\item \label{thmB2} $\dfrac{\delta_i(\delta_j e)}{\delta_j e}
	\in\widetilde{K}$.
\end{enumerate}
When $|\Delta|=2$, one may take $\widetilde{K}=K$.
\end{maintheorem}

It is worth noting that, for autonomous planar systems over $\mathbb{C}$,
conditions \eqref{thmB1} and \eqref{thmB2} can be obtained from the
integrating factor of Singer's theorem, and conversely, the integrating factor can be obtained from
\eqref{thmB1} and \eqref{thmB2}. Suppose $f_1f_2\neq0$. Let $\lambda$ be
the integrating factor provided by Singer's theorem and $e$ be the associated
first integral. Then $\delta_1e=\lambda f_2$ and
$\delta_2e=-\lambda f_1$. Hence $\delta_1e/\delta_2e=-f_2/f_1\in K$ and
\begin{equation*}
	\frac{\delta_i(\delta_1e)}{\delta_1e}
	=\frac{\delta_i\lambda}{\lambda}+\frac{\delta_if_2}{f_2}\in K,
	\qquad
	\frac{\delta_i(\delta_2e)}{\delta_2e}
	=\frac{\delta_i\lambda}{\lambda}+\frac{\delta_if_1}{f_1}\in K
	\qquad (i=1,2),
\end{equation*}
which are \eqref{thmB1} and \eqref{thmB2} with $\widetilde K=K$.
Conversely, if $e$ is as in Theorem~\ref{main-thm-B} with
$\widetilde K=K$, then $\lambda:=-\delta_2e/f_1$ satisfies
$\delta_1e=\lambda f_2$ and $\delta_2e=-\lambda f_1$ and
$\delta_i\lambda/\lambda=\delta_i(\delta_2e)/\delta_2e-\delta_if_1/f_1\in K$
by \eqref{thmB2}, so $\lambda=\exp\int(U\,dy_1+V\,dy_2)$ with
$U:=\delta_1\lambda/\lambda$, $V:=\delta_2\lambda/\lambda\in K$ is an
integrating factor of $f_2\,dy_1-f_1\,dy_2$, as in Singer's theorem. 

We stress that conditions \eqref{thmB1} and \eqref{thmB2} of
Theorem~\ref{main-thm-B} are obtained without any prior description of
the integrating factor. They are consequences of the structure of the
differential Galois group of a Liouvillian Picard-Vessiot extension
containing $e$, and the integrating factor is then read off from them.

For autonomous systems in $n$ variables, an analogous computation relates
Theorem~\ref{main-thm-B} to \cite[Theorem~3]{ACPW25}
(see Proposition~\ref{prop-dictionary}). When $n>2$, one cannot in general take $\widetilde K=K$ in Theorem~\ref{main-thm-B}; an example of
such an autonomous system appears in \cite[Section 5.2]{CPW25}. We show in Proposition~\ref{prop-descent} that if condition \eqref{thmB1} holds over $K$, then $e$ can be replaced by a Liouvillian first integral $\tilde e$ satisfying \eqref{thmB1} and \eqref{thmB2} over $K$, so that one may take $\widetilde K=K$ in Theorem~\ref{main-thm-B}. If $|\Delta|=2$, then the equation $De=0$ forces the ratio of the two
derivatives of any first integral $e$ to lie in $K$, so \eqref{thmB1} holds
automatically---this is how the last assertion of
Theorem~\ref{main-thm-B} is obtained.

Finally, we show that Theorem~\ref{main-thm-A} is sharp. We exhibit a
system admitting a first integral in a
Picard-Vessiot extension whose differential Galois group is
$\mathrm{SL}_2(\mathbb C)$, but admitting no first integral in its
Picard-Vessiot ring. Therefore, the Liouvillian
hypothesis in Theorem~\ref{main-thm-A} cannot be dropped. In a second
example, we exhibit a system having first integrals in a Picard-Vessiot
ring, none of which is Liouvillian. Thus, a first integral in a Picard-Vessiot ring need not be Liouvillian.

%%%%%%%%%%%%%%%%%%%%%%%%%%%%%%%%%%%%%%%%%%%%%%%%%%%%%%%%%%%%%%%%%%%%%%%%%%%%%
\section{Picard-Vessiot theory of partial differential fields}\label{prelim}
%%%%%%%%%%%%%%%%%%%%%%%%%%%%%%%%%%%%%%%%%%%%%%%%%%%%%%%%%%%%%%%%%%%%%%%%%%%%%

In this section, we briefly recall the Picard-Vessiot theory of $\Delta$-fields following \cite{Kol52,vdPS03}. Throughout this paper, all rings are assumed to be commutative with identity and to contain the field of rational numbers $\mathbb{Q}$. 

A \emph{$\Delta$-ring} is a ring equipped with a finite set $\Delta$ of commuting derivations. If a $\Delta$-ring is a field, it is called a \emph{$\Delta$-field}. The \emph{ring of constants} of a $\Delta$-ring $R$ is defined as
\begin{equation*}
    C(R,\Delta) := \{a \in R : \delta(a) = 0 \text{ for all } \delta \in \Delta\}.
\end{equation*} 
If $R$ is a $\Delta$-field, then $C(R, \Delta)$ is a field. The notions of  $\Delta$-subring, $\Delta$-ring extension, $\Delta$-ideal, $\Delta$-homomorphism, and $\Delta$-automorphism are defined naturally (see \cite[Chapter 1]{Kol73}). When $\Delta = \{D\}$ consists of a single derivation, we use the standard prefix $D$ instead of $\{D\}$, writing $D$-ring, $D$-ideal, and $C(R,D)$. For a $\Delta$-field extension $E$ of a $\Delta$-field $K$, the group of $\Delta$-automorphisms 
of $E$ that fix each element of $K$ is denoted by $\mathscr{G}(E|K)$ and is 
called the \emph{differential Galois group} of $E$ over $K$. These automorphisms are
called $K$-$\Delta$-automorphisms.  The free commutative monoid generated by $\Delta$ is denoted by $\Theta$.

For the rest of this section, we fix a $\Delta$-field $K$ with $\Delta = \{\delta_1, \dots, \delta_r\}$, and assume that its field of constants $C := C(K,\Delta)$ is algebraically closed. A system of linear differential equations
\begin{equation}\label{lin-sys-equ}
    \delta_i Y = A_i Y, \quad A_i \in \mathrm{M}_m(K), \quad i = 1, \dots, r,
\end{equation}
is said to be \emph{integrable} if $\delta_i A_j - \delta_j A_i = A_i A_j - A_j A_i$ for all $1 \le i, j \le r$. A \emph{Picard-Vessiot extension} $E$ of the $\Delta$-field $K$ for an integrable system \eqref{lin-sys-equ} is a $\Delta$-field extension of $K$ such that $C(E,\Delta)=C$, there is a \emph{fundamental solution matrix} $Z\in \mathrm{GL}_m(E)$ for the system, that is, $\delta_i(Z)=A_iZ$ for all $1\leq i\leq r$, and $E$ is the field of fractions of the $K$-algebra $R$ generated by the entries of $Z$ and $1/\mathrm{det}(Z)$. The $K$-algebra $R$ is a $\Delta$-ring called  the \emph{Picard-Vessiot ring} over $K$ for the integrable system \eqref{lin-sys-equ}.  

%is a $\Delta$-ring extension $R$ of $K$ satisfying the following properties: \begin{enumerate}[(i)] \item $R$ is $\Delta$-simple; that is, the only $\Delta$-ideals of $R$ are $(0)$ and $R$. \item There exists a matrix $Z \in \mathrm{GL}_m(R)$, called a \emph{fundamental solution matrix}, such that $\delta_i Z = A_i Z$ for all $i = 1, \dots, r$. \item $R$ is generated as a $K$-algebra by the entries of $Z$ and $1/\det Z$.\end{enumerate}

Under the assumption that $C$ is algebraically closed, Picard-Vessiot extensions always exist for an integrable system and are unique up to  $K$-$\Delta$-isomorphisms. There is an intrinsic characterisation of a Picard-Vessiot ring:

\begin{theorem}\label{char-PV-ring} Let $E$ be a Picard-Vessiot extension of $K$, $R$ be its Picard-Vessiot ring, and $f\in E$. The
following are equivalent.
\begin{enumerate}[(i)]
	\item $f\in R$.
	\item The $C$-vector space $V_f:=\mathrm{Span}_C\{\sigma(f):\sigma\in\mathscr{G}(E|K)\}$
	is finite-dimensional.
	\item $f$ lies in a finite-dimensional $\Delta$-stable $K$-subspace of
	$E$.
\end{enumerate}
\end{theorem}

\begin{proof} This follows, mutatis mutandis, from the proof of \cite[Corollary 1.38]{vdPS03}.\end{proof}

\subsection{Differential Galois correspondence}\label{galoiscorrespondence}
Let $E$ be a Picard-Vessiot extension of $K$ for the integrable system \eqref{lin-sys-equ}. The differential Galois group $\mathscr{G}(E|K)$ possesses the structure of a linear algebraic group defined over $C$. Indeed, if $Z$ is a fundamental solution matrix of the system, then for any $\sigma \in \mathscr{G}(E|K)$, the matrix $\sigma(Z)$ is also a fundamental solution matrix, implying that $\sigma(Z) = Z C_\sigma$ for a unique $C_\sigma \in \mathrm{GL}_m(C)$. The faithful representation $\mathscr{G}(E|K) \to \mathrm{GL}_m(C)$ given by $\sigma \mapsto C_\sigma$ realises $\mathscr{G}(E|K)$ as a Zariski-closed subgroup of $\mathrm{GL}_m(C)$.

As in classical  Galois theory, there is a Galois correspondence between the intermediate
$\Delta$-fields $K\subseteq L\subseteq E$ and the algebraic subgroups
$\mathscr{H}$ of $\mathscr{G}(E|K)$ (see \cite{Kol52} and
\cite[Proposition~1.34 and Appendix~D]{vdPS03}). It is given by the maps
\begin{align*}
L &\longmapsto \mathscr{G}(E|L) := \{ \sigma \in \mathscr{G}(E|K) : \sigma(a)=a \text{ for all } a \in L \}, \\
\mathscr{H} &\longmapsto E^{\mathscr{H}} := \{ a \in E : \sigma(a)=a \text{ for all } \sigma \in \mathscr{H} \}.
\end{align*}
Under this correspondence, an algebraic subgroup $\mathscr{H}$ is normal if
and only if $E^{\mathscr{H}}$ is a Picard-Vessiot extension of $K$, and in
that case $\mathscr{G}(E^{\mathscr{H}}|K)\cong\mathscr{G}(E|K)/\mathscr{H}$.
The relative algebraic closure of $K$ in $E$ is precisely the fixed field
$E^{\mathscr{G}(E|K)^0}$, where $\mathscr{G}(E|K)^0$ denotes the identity
component of $\mathscr{G}(E|K)$. Moreover, $E^{\mathscr{G}(E|K)^0}$ is a
finite algebraic Galois extension of $K$ whose Galois group is isomorphic to
the finite group $\mathscr{G}(E|K)/\mathscr{G}(E|K)^0$.

\subsection{The complete Picard-Vessiot closure}
In \cite{Mag01,Mag02,Mag22}, Magid developed the theory of Picard-Vessiot closures for fields equipped with a single derivation. His constructions and proofs generalise directly to the case of a finite set of commuting derivations. For convenience, we shall only recall few results that are relevant to this article.

%A \emph{Picard-Vessiot closure} of $K$ is a $\Delta$-field extension $K_1$ satisfying the following properties: \begin{enumerate}[(i)] \item $C(K_1,\Delta) = C$;  \item $K_1$ is the union of its Picard-Vessiot subextensions, i.e. $$K_1=\bigcup\{L: K\subseteq L\subseteq K_1, L \text{ is a Picard-Vessiot extension of } K\};$$ \item every Picard-Vessiot extension $L$ of $K$ admits a $K$-$\Delta$-embedding into $K_1$. \end{enumerate}

%Such an extension exists and is unique up to a $K$-$\Delta$-isomorphism (see \cite[Theorem, p.~6]{Mag01} and \cite[Theorems~11 and~17]{Mag02}).  Since $K_1$ is a $\Delta$-field whose field of constants $C$ is algebraically closed, the construction can be iterated. Set $K_0 := K$, and for each $i \ge 1$, let $K_i$ be a Picard-Vessiot closure of $K_{i-1}$. The  union
%\[K_\infty := \bigcup_{i \ge 0} K_i\] is a $\Delta$-field called the \emph{complete Picard-Vessiot closure} of $K$. Although each intermediate field $K_i$ is determined only up to $K_{i-1}$-$\Delta$-isomorphism, their union $K_\infty$ is unique up to $K$-$\Delta$-isomorphism (\cite[\S2]{Mag22}). The field $K_\infty$ can be characterised intrinsically by the following theorem.

\begin{theorem} {{\cite[Theorem~1]{Mag22}}}\label{thm-K-inf}
There exists a $\Delta$-field extension $K_\infty$ of $K$ having the following properties:
\begin{enumerate}[(i)]
\item $C(K_\infty,\Delta)=C$;
\item every integrable system with coefficients in $K_\infty$ has a
  fundamental solution matrix with entries in $K_\infty$;
\item if $E$ is a $\Delta$-field such that $K\subseteq E\subseteq K_\infty$ and every integrable system with coefficients in
  $E$ has a fundamental solution matrix with entries in $E$, then
  $E=K_\infty$.
\end{enumerate}
Moreover, any $\Delta$-field extension of $K$ with these three
properties is $K$-$\Delta$-isomorphic to $K_\infty$.
\end{theorem}
The $\Delta$-field extension $K_\infty$ is called a \emph{complete Picard-Vessiot closure} of $K$. We note that if $E$ is a $\Delta$-field with $K\subseteq E\subseteq K_\infty$,
then $E_\infty=K_\infty$ \cite[Corollary~1]{Mag22}. Let $E\subseteq K_\infty$ be a $\Delta$-field extension of $K$ which is not necessarily a finitely generated field extension of $K$. If every element of
$\mathscr G(K_\infty|K)$ restricts to an automorphism of $E$, then $E$ is
called a \emph{normal} extension of $K$. Any Picard-Vessiot extension of $K$
is a normal extension of $K$, and if $E$ is a normal extension of $K$, then
its algebraic closure $\overline{E}$ in $K_\infty$ is also a normal extension
of $K$. The next theorem generalizes \cite[Proposition~1]{Mag22} in its first part, while the second part follows from the fundamental theorem of the complete Picard–Vessiot closure.

\begin{theorem}\label{lemma-aut-extend}
\begin{enumerate}[(i)]
\item Let $L$ be a $\Delta$-subfield of $K_\infty$ containing $K$. Every
$K$-$\Delta$-embedding $\tau\colon L\to K_\infty$ extends to a
$K$-$\Delta$-automorphism of $K_\infty$.
\item If $E\subseteq K_\infty$ is a normal extension of $K$, then
$E^{\mathscr{G}(E|K)}=K$. In particular, if $L$ is a Picard-Vessiot
extension of $K$ and $\overline{L}$ is its algebraic closure in $K_\infty$,
then $\overline{L}^{\;\mathscr{G}(\overline{L}|K)}=K$.
\end{enumerate}
\end{theorem}

\begin{proof}
(i) Since $L_\infty=K_\infty$, the remark following
\cite[Theorem~3]{Mag22} shows that $\tau$ extends to a
$K$-$\Delta$-embedding $\sigma\colon K_\infty\to K_\infty$. The image
$\sigma(K_\infty)$ is a $\Delta$-field between $K$ and $K_\infty$ in which every integrable system has a fundamental solution matrix, so
$\sigma(K_\infty)=K_\infty$ by Theorem~\ref{thm-K-inf}(iii).

(ii) By \cite[Theorem~5]{Mag22}, $K_\infty^{\mathscr{G}(K_\infty|K)}=K$.
Since $E$ is normal, every element of $\mathscr{G}(K_\infty|K)$ restricts to
an element of $\mathscr{G}(E|K)$, and therefore
$E^{\mathscr{G}(E|K)}\subseteq K_\infty^{\mathscr{G}(K_\infty|K)}=K$. The
last assertion follows because $L$, and hence $\overline{L}$, is a normal
extension of $K$.
\end{proof}

\subsection{Liouvillian extensions} \label{defn-LiouvExtn}
A $\Delta$-field extension $E$ of $K$ with $C(E,\Delta)=C$ is called a \emph{Liouvillian extension} if there exists a tower of $\Delta$-fields $$K = E_0 \subset E_1 \subset \cdots \subset E_{\ell} = E$$ such that for each $i=1,\dots, \ell$, $E_i = E_{i-1}(t_i)$ and one of the following holds:
\begin{enumerate}[(i)]
    \item $t_i$ is algebraic over $E_{i-1}$;
    \item $\delta(t_i) \in E_{i-1}$ for all $\delta \in \Delta$;
    \item $t_i \neq 0$ and $\delta(t_i)/t_i \in E_{i-1}$ for all $\delta \in \Delta$.
\end{enumerate}

Note that $E_i$ is a Picard-Vessiot extension of $E_{i-1}$ with a differential Galois group isomorphic to subgroups of $\mathbb{G}_a$ or $\mathbb G_m$, whenever (ii) or (iii) holds respectively.
A \emph{Liouvillian Picard-Vessiot extension} of $K$ is a Picard-Vessiot extension of $K$ that is also a Liouvillian extension.  Such extensions can be characterised in terms of their differential Galois groups.

\begin{theorem}{{\cite[Theorem~4]{Kol52}}}\label{Liou_PV}
A Picard-Vessiot extension $E$ of $K$ is Liouvillian if and only if the
identity component $\mathscr{G}(E|K)^0$ of its differential Galois group is
solvable.
\end{theorem}

Let $K\subseteq E\subseteq K_\infty$ be an intermediate $\Delta$-field, 
and let $L$ be a Picard-Vessiot extension of $E$ for an integrable 
system over $E$. By Theorem~\ref{thm-K-inf}, this system admits a 
fundamental solution matrix $Z\in\operatorname{GL}_m(K_\infty)$. The 
field $E(Z)$, generated over $E$ by the entries of $Z$, is a 
Picard-Vessiot extension of $E$ contained in $K_\infty$; it is 
therefore $E$-$\Delta$-isomorphic to $L$. Consequently, every 
Picard-Vessiot extension of $E$ embeds over $E$ into $K_\infty$. 
This allows us to conclude that any tower of $\Delta$-fields 
$K=M_0\subset M_1\subset\cdots\subset M_\ell$, where each $M_i$ 
is a Picard-Vessiot extension of $M_{i-1}$, can be embedded in $K_\infty$.

Since any finite algebraic extension embeds into a finite Galois extension, 
a Liouvillian extension of $K$ can always be embedded into a tower 
$$K = F_0 \subset F_1 \subset \cdots \subset F_p$$
of $\Delta$-fields where each $F_i\supseteq F_{i-1}$ is a Liouvillian 
Picard-Vessiot extension. It follows that every Liouvillian extension 
of $K$ likewise embeds into $K_\infty$. Accordingly, throughout the 
remainder of this paper, we shall assume that all Picard-Vessiot 
extensions and Liouvillian extensions of $K$ are $\Delta$-subfields of~$K_\infty$.

%%%%%%%%%%%%%%%%%%%%%%%%%%%%%%%%%%%%%%%%%%%%%%%%%%%%%%%%%%%%%%%%%%%%%%%%%%%%%
	\section{First integrals in Liouvillian Picard-Vessiot extensions}\label{sec-in-ring}
	%%%%%%%%%%%%%%%%%%%%%%%%%%%%%%%%%%%%%%%%%%%%%%%%%%%%%%%%%%%%%%%%%%%%%%%%%%%%%

In this section,  $k$, $C$, $K$, $\Delta$ and $D$ are as defined in Section \ref {intro}.  If \eqref{system-fode} has a first integral in $K$, then Theorems~\ref{main-thm-A} and~\ref{main-thm-B} hold trivially. Therefore, we assume that
$C(K,D)=C$, that is, the system \eqref{system-fode} has no rational first integrals; in particular, $f_i \neq 0$ for each $i=1,2,\dots, n$.

Let $E$ be a $\Delta$-field extension of $K$. Then $E$ is also a
$D$-field extension of $K$, and every $K$-$\Delta$-automorphism of $E$
is a $K$-$D$-automorphism. The field $C(E,D)$ contains $C(E,\Delta)$ and may be strictly larger. The elements of $C(E,D)\setminus C(E,\Delta)$  are precisely the first integrals of \eqref{system-fode} in $E$. Since any derivation on $K$ extends uniquely to an algebraic extension $\widetilde K$, the field $\widetilde K$ inherits a unique structure as both a $\Delta$-field  and a $D$-field. 

\begin{lemma}\label{lemma-setup}
Let $K\subseteq F\subset K_\infty$ be $\Delta$-fields, $E$ be a
Picard-Vessiot extension of $F$, $R\subset E$ be its Picard-Vessiot ring and
$\mathscr{G}:=\mathscr{G}(E|F)$. Then, for any $e\in C(E,D)$, there exist a
nonzero $w\in R$ and $\alpha\in F$ such that $we\in R$ and $Dw=\alpha w$.
\end{lemma}

\begin{proof}
Since $E$ is the field of fractions of $R$, the set
\begin{equation*}
	I:=\{a\in R : ae\in R\}
\end{equation*}
is a nonzero ideal of $R$. As $R$ is a $D$-ring, we have
$D(ae)=(Da)e\in R$ for $a\in I$. Hence $I$ is a $D$-ideal.

Fix a nonzero $v\in I$. By Theorem~\ref{char-PV-ring}, the orbit
$\{\sigma(v):\sigma\in\mathscr{G}\}$ spans a finite-dimensional
$C$-subspace of $R$. Let
$$V=\mathrm{Span}_{C(E,D)}\{\sigma(v):\sigma\in\mathscr{G}\}.$$
Every $\sigma\in\mathscr{G}$ commutes with $D$ and hence maps $C(E,D)$ onto
itself, so $V$ is a finite-dimensional $\mathscr{G}$-stable
$C(E,D)$-subspace of $E$.

We may choose a $C(E,D)$-basis $\{v_1,\dots,v_m\}$ of $V$ such that
$v_1,\dots,v_m\in R$ and $v_1=v$. Let $w:=\mathrm{wr}_D(v_1,\dots,v_m)$ be
the Wronskian of $v_1,\dots, v_m$ with respect to the derivation $D$, that
is, $w=\det\big(D^{i-1}v_j\big)$. Then $w\in R$ and since $v_1,\dots,v_m$
are linearly independent over the field $C(E,D)$, we have $w\neq0$. Observe
that $V$ is the space of solutions in $E$ of $\mathcal{L}(Y)=0$, where
$\mathcal{L}$ is the monic operator
\begin{equation*}
	\mathcal{L}(Y):=\frac{\mathrm{wr}_D(v_1,\dots,v_m,Y)}{w}
	=D^mY+a_{m-1}D^{m-1}Y+\cdots+a_0Y,\quad a_0,\dots,a_{m-1}\in E,
\end{equation*}
and it is well known that $D(w)=-a_{m-1}w$. We shall now show that
$a_{m-1}\in F$. Since $V$ is $\mathscr{G}$-stable, the elements of $V$ are
also solutions of the operator
$$\mathcal L_{\sigma}(Y):=\left(\sigma(a_{m-1})-a_{m-1}\right)D^{m-1}Y+\cdots
+\left(\sigma(a_0)-a_0\right)Y$$
for each $\sigma\in \mathscr{G}$. Since the order of $\mathcal L_\sigma$ is
at most $m-1$ and $V$ is of dimension $m$ over $C(E,D)$, we must have
$\mathcal L_\sigma=0$. That is,
$a_0,\dots,a_{m-1}\in E^{\mathscr{G}}=F$. Finally, expanding
$\det(D^{i-1}v_j)$ along the first column
$(v,Dv,\dots,D^{m-1}v)^{\mathsf t}$, whose entries lie in the $D$-ideal
$I$, shows $w\in I$. Hence $we\in R$, and $Dw=\alpha w$ with
$\alpha:=-a_{m-1}\in F$.
\end{proof}

\subsection{Proof of Theorem \ref{main-thm-A} and some consequences} 
We are now in a position to prove Theorem \ref{main-thm-A}. Let $K\subseteq F\subset E\subset K_\infty$ be $\Delta$-fields, $E$ be a
Liouvillian Picard-Vessiot extension of $F$ with Picard-Vessiot ring $R$
over $F$ and $\mathscr{G}:=\mathscr{G}(E|F)$. We will prove that
$C(E,D)\neq C$ implies $C(R,D)\neq C$ and from this the proof of
Theorem~\ref{main-thm-A} follows by taking $F=K$.

If $C(F,D)\neq C$ there is nothing to prove, so assume $C(F,D)=C$ and let
$e\in C(E,D)\setminus C$. Then Lemma~\ref{lemma-setup} provides
$w\in R\setminus\{0\}$ and $\alpha\in F$ with $we\in R$ and $Dw=\alpha w$.
From the Galois correspondence, $E^{\mathscr{G}^0}$ is a $\Delta$-field
which is also a finite (Galois) extension of $F$. In particular,
$E^{\mathscr{G}^0}$ is a finite-dimensional $\Delta$-stable $F$-subspace of
$E$ and therefore $E^{\mathscr{G}^0}\subset R$ by
Theorem~\ref{char-PV-ring}.

Let $V=\mathrm{Span}_C\{\sigma(w) : \sigma\in \mathscr{G}^0\}$. Since
$w\in R$, Theorem~\ref{char-PV-ring} shows that $V$ is a finite-dimensional
$\mathscr{G}^0$-stable $C$-subspace of $R$. Note that
$D(\sigma(w))=\sigma(D(w))=\alpha \sigma(w)$ and therefore
$\sigma(w)/w\in C(E,D)$ for every $\sigma\in\mathscr{G}^0$. Thus, every
element of $V$ has the form $cw$ with $c\in C(E,D)$.

By Theorem~\ref{Liou_PV}, $\mathscr{G}^0$ is a connected solvable algebraic
group, so by the Lie-Kolchin theorem, there exists a $C$-basis
$e_1w,\dots, e_mw$ of $V$, with $e_1,\dots,e_m\in C(E,D)$, such that
\begin{equation}\label{solvable-group-action}
\sigma(e_iw)=\sum^i_{j=1}c_{ij\sigma}e_jw\, ,\quad (c_{ij\sigma})\in
\mathrm{GL}_m(C)\ \text{and } c_{ij\sigma}=0\ \text{for all } i<j.
\end{equation}
for each $\sigma\in \mathscr{G}^0.$ Now we divide the proof into two cases.

\emph{Case} 1: Let $\mathrm{dim}_CV=1$. Then $V$ is spanned by $w$ and we
obtain $\sigma w=c_\sigma w$, where $c_\sigma\in C$, for all
$\sigma\in\mathscr{G}^0$. Therefore, for each $\delta\in\Delta$, the element
$\delta (w)/w$ is fixed by $\mathscr{G}^0$ and we obtain
$\delta(w^{-1})=\beta w^{-1}$ with
$\beta:=-\delta(w)/w\in E^{\mathscr{G}^0}$. Hence, the $F$-space
$W:=E^{\mathscr G^0}+ E^{\mathscr G^0}w^{-1}$, containing
$E^{\mathscr G^0}$ and $w^{-1}$, is a finite-dimensional $\Delta$-stable
$F$-vector space. By Theorem~\ref{char-PV-ring}, $w^{-1}\in R$. Since
$we\in R$, we obtain $e=w^{-1}we\in R$. Thus, in this case,
$C(R,D)\neq C$.

\emph{Case} 2: Let $\mathrm{dim}_CV\geq 2$. We have
$\sigma(e_1w)=c_{11\sigma} e_1w$ and
$\sigma(e_2w)=c_{21\sigma} e_1w+c_{22\sigma}e_2w$. Therefore
$$\sigma\left (\frac{e_2}{e_1}\right)=\frac{c_{22\sigma}}{c_{11\sigma}}
\frac{e_2}{e_1}+\frac{c_{21\sigma}}{c_{11\sigma}}$$
and this implies that the $C$-vector space $\mathrm{Span}_C\{1,e_2/e_1\}$
is $\mathscr{G}^0$-stable. Since $\mathscr{G}^0$ has finite index in
$\mathscr{G}$, the $\mathscr{G}$-orbit of $e_2/e_1$ also spans a
finite-dimensional $C$-space, and therefore, again by
Theorem~\ref{char-PV-ring}, the $D$-constant $e_2/e_1$ belongs to $R$.
Furthermore, since $e_1w$ and $e_2w$ are $C$-linearly independent,
$e_2/e_1\notin C$. Thus, we obtain $C(R,D)\neq C$ in this case as well. This completes the proof of Theorem \ref{main-thm-A}.\qed
    
	We now use the solvability of the identity component of the Galois group once again to describe a Liouvillian first integral in the Picard-Vessiot ring of a Liouvillian Picard-Vessiot extension.

	\begin{corollary}\label{thm-in-Liouv-PV}
Let $K\subseteq F\subset E\subset K_\infty$ be $\Delta$-fields with
$C(F,D)=C$, where $E$ is a Liouvillian Picard-Vessiot extension of $F$ with
Picard-Vessiot ring $R$, and let $\mathscr U$ be the unipotent radical of
$\mathscr{G}:=\mathscr{G}(E|F)$. Suppose that \eqref{system-fode} admits a first integral in $E$. Then there is a first integral $e\in R$ such that,
for all $\delta_i,\delta_j\in\Delta$ with $\delta_j(e)\neq 0$,
$$\frac{\delta_i(e)}{\delta_j(e)}\in E^{\mathscr{G}^0}\quad \text{and}\quad
\frac{\delta_i(\delta_j(e))}{\delta_j(e)}\in E^{\mathscr{G}^0}.$$
Furthermore, if \eqref{system-fode} admits a first integral in
$E^{\mathscr U}$, then $e$ can be chosen in $E^{\mathscr U}\cap R$ with
$\delta_i(e)/e\in E^{\mathscr G^0}$ for all $\delta_i\in \Delta$.
\end{corollary}

\begin{proof}
Recall that $E^{\mathscr U}$ is a Liouvillian Picard-Vessiot extension of
$F$ with Picard-Vessiot ring $E^{\mathscr U}\cap R$, and that it is
$\mathscr{G}$-stable.

\emph{Case} 1: Suppose that \eqref{system-fode} admits a first integral in
$E^{\mathscr U}$. By the proof of Theorem~\ref{main-thm-A}, there is
$\tilde{e}\in C(E^{\mathscr U}\cap R,D)\setminus C$. Let
$V=\mathrm{Span}_C\{\sigma(\tilde{e}) : \sigma\in \mathscr{G}^0\}$, which is
finite-dimensional by Theorem~\ref{char-PV-ring}. Then $V\subset E^{\mathscr U}\cap R$. Since $E^{\mathscr U}$ is a
Picard-Vessiot extension of $E^{\mathscr{G}^0}$ whose Galois group
$\mathscr{G}^0/\mathscr U$ is a torus, $V$ has a $C$-basis
$\{e_1,\dots,e_m\}$ with $\sigma(e_j)=c_{j\sigma}e_j$, $c_{j\sigma}\in C$,
for all $\sigma\in\mathscr{G}^0$. Since $\tilde e\in V\setminus C$, there is
a $j$ with $e_j\notin C$. Take $e:=e_j$. Then
$a_i:=\delta_i(e)/e\in E^{\mathscr{G}^0}$ for all $\delta_i\in \Delta$, and
whenever $\delta_j(e)\neq 0$ we have $a_j\neq 0$,
$\delta_i(e)/\delta_j(e)=a_i/a_j\in E^{\mathscr{G}^0}$ and
$\delta_i(\delta_j(e))/\delta_j(e)=\delta_i(a_j)/a_j+a_i\in E^{\mathscr{G}^0}$.

\emph{Case} 2: $C(E^{\mathscr U},D)=C$. By the proof of
Theorem~\ref{main-thm-A}, there is $\tilde{e}\in C(R,D)\setminus C$. Let
$V=\mathrm{Span}_C\{\sigma(\tilde{e}) : \sigma\in \mathscr{G}^0\}
\subseteq C(R,D)$, which is finite-dimensional by
Theorem~\ref{char-PV-ring}. By the Lie-Kolchin theorem there is a $C$-basis
$\{e_1,\dots,e_m\}$ of $V$ such that
\begin{equation}\label{triangular-action}
\sigma(e_i)=\sum^i_{j=1}c_{ij\sigma}e_j,\quad (c_{ij\sigma})\in
\mathrm{GL}_m(C)\ \text{and } c_{ij\sigma}=0\ \text{for all } i<j.
\end{equation}
for each $\sigma\in \mathscr{G}^0.$ The map $\sigma\mapsto c_{11\sigma}$ is a character of $\mathscr{G}^0$.
Since there are no nontrivial morphisms from the unipotent group
$\mathscr U$ to $\mathbb{G}_m$, its restriction to $\mathscr U$ is trivial.
Thus, $e_1\in E^{\mathscr U}$, and so $e_1\in C(E^{\mathscr U},D)=C$; in
particular $c_{11\sigma}=1$ for all $\sigma\in\mathscr{G}^0$. Since
$\tilde{e}\in V\setminus C$, we have $m\geq 2$, and by $C$-linear
independence $e:=e_2/e_1\in R$ is not in $C$. For any
$\sigma\in \mathscr{G}^0$ we have $\sigma(e)=c_{22\sigma}e+c_{21\sigma}$,
and hence $\sigma(\delta(e))=c_{22\sigma}\delta(e)$ for all
$\delta\in \Delta$. Therefore, for all $\delta_i,\delta_j\in\Delta$ with
$\delta_j(e)\neq 0$,
$$\frac{\delta_i(e)}{\delta_j(e)}\in E^{\mathscr{G}^0}\quad\text{and}\quad
\frac{\delta_i(\delta_j(e))}{\delta_j(e)}\in E^{\mathscr{G}^0}.$$
\end{proof}

We now give two examples of planar systems illustrating the situations
described in Corollary~\ref{thm-in-Liouv-PV}. In both examples,
$k=\mathbb{C}$, $K=\mathbb{C}(y_1,y_2)$ and $F=K$. Since
$\delta_1e/\delta_2e=-f_2/f_1\in K$ for any first integral $e$, the first
condition of Corollary~\ref{thm-in-Liouv-PV} holds trivially.
\begin{example}\label{example-dichotomy}
	 \emph{(a)} Fix
	$\mu\in\mathbb{C}\setminus\mathbb{Q}$ and consider the system \begin{equation*}
		\begin{cases}
				y_1' = y_1\\
				y_2' = \mu y_2.
		\end{cases}
	\end{equation*} 
	Here $D=y_1\delta_1+\mu y_2\delta_2$. Let $E$ be a Picard-Vessiot extension of $K$ for the integrable system 
	$\delta_1Y=(-\mu/y_1)Y$, $\delta_2Y=0$ and $u\in E$ be a nonzero solution of the system. Then, $E=K(u)$ and $E$ is a Liouvillian
	Picard-Vessiot extension of $K$ with Picard-Vessiot ring
	$R=K[u,u^{-1}]$.  Note that $e:=y_2u$ lies in $R$, $De=0$
	and
	$$
	\frac{\delta_1e}{e}=-\frac{\mu}{y_1}\in K, \qquad
	\frac{\delta_2e}{e}=\frac{1}{y_2}\in K.
	$$

\emph{(b)} Consider the system
\begin{equation*}
	\begin{cases}
		y_1' = 1\\
		y_2' = 2y_1y_2+1.
	\end{cases}
\end{equation*}
Here $D=\delta_1+(2y_1y_2+1)\delta_2$. Let $E$ be the Picard-Vessiot
extension of $K$ for the integrable system
$$\delta_1Y
=\begin{pmatrix} 0 & 1\\ 0 & -2y_1\end{pmatrix}
Y,
\qquad
\delta_2Y=0 ,
$$
with Picard-Vessiot ring $R$. Since $(1,0)^{\mathsf t}$ is a solution, we
may choose a fundamental solution matrix of the form
$Z=\left(\begin{smallmatrix} t & 1\\ z & 0\end{smallmatrix}\right)$ with
$z\neq0$, so that\footnote{The elements $z$ and $t$ can be thought of as
$\exp(-y^2_1)$ and $\int \exp(-y^2_1)\,dy_1$.}
$$\delta_1(t)=z,\quad \delta_1(z)=-2y_1z\quad \text{and}\quad
\delta_2(t)=\delta_2(z)=0.$$
Clearly, $E=K(z,t)$ and it is therefore a Liouvillian Picard-Vessiot
extension of $K$. Let $e:=y_2z-t\in R$ and observe that
$\delta_1e=-(2y_1y_2+1)z$ and $\delta_2e=z$. Therefore $De=0$ and $e$ is
not a $\Delta$-constant. Moreover,
$$
\frac{\delta_1\delta_2e}{\delta_2e}=\frac{\delta_1z}{z}=-2y_1\in K,
\qquad
\frac{\delta_2\delta_2e}{\delta_2e}=0, \qquad
\frac{\delta_1e}{\delta_2e}=-(2y_1y_2+1)\in K ,
$$
and the same holds with $\delta_1e$ in place of $\delta_2e$. Note that the
element $\delta_2e/e=z/(y_2z-t)$ is transcendental over $K$, hence in
particular, does not lie in a finite algebraic extension of $K$. The tower
of $\Delta$-fields $K\subset K(z)\subset K(z,t)$ is exactly the extension
described after Singer's theorem in Section~\ref{intro}: an exponential of
an integral followed by an integral. In the notation of Singer's theorem,
the integrating factor is $\lambda=-z=-\exp(-y_1^{2})$, with $U=-2y_1$ and
$V=0$.
\end{example}

\section{First integrals in Liouvillian extensions}\label{sec-liouv-first-int}

In the next three propositions, $K$ denotes an arbitrary $\Delta$-field with
$\Delta=\{\delta_1,\dots,\delta_r\}$ whose field of constants $C$ is
algebraically closed, as in Section~\ref{prelim}. The first two propositions partly generalise Theorem~\ref{char-PV-ring}.

\begin{proposition} \label{remark-orbit-general} Let $L$ be a $\Delta$-field extension of $K$, which is possibly infinitely generated as a field over $K$, with $C(L,\Delta)=C$. Suppose that $L^{\mathscr G(L|K)}=K$. If $L$ contains  a finite-dimensional $\mathscr G(L|K)$-stable $C$-subspace $V$ with a $C$-basis $\{f_1,\dots,f_s\}$, then there are operators $\theta_1,\dots,\theta_s\in \Theta$ and an integrable system $$\delta_i(Y)=A_iY,\quad A_i\in \mathrm{M}_s(K),\qquad i=1,\dots, r$$ having $(\theta_if_j)\in \mathrm{GL}_s(L)$ as a fundamental solution matrix. In particular, there is a Picard-Vessiot ring $R$ over $K$ such that $V\subset R\subset L$.
\end{proposition}

\begin{proof} For $\eta_1,\dots,\eta_s\in\Theta$, put
$\mathrm{wr}_{\eta_1,\dots,\eta_s}(f_1,\dots,f_s):=\det(\eta_if_j)$. Since $V$ is $\mathscr G(L|K)$-stable, for each $\sigma\in \mathscr G(L|K)$, there is a matrix $C_\sigma:=(c_{ij\sigma})\in \mathrm{GL}_s(C)$ such that \begin{equation}\label{G-actiononsolution}\sigma(f_i)=\sum^s_{j=1}c_{ji\sigma}f_j.\end{equation}

Since every element of $\mathscr G(L|K)$ commutes with every $\eta\in\Theta$, we have
$$\sigma(\mathrm{wr}_{\eta_1,\dots,\eta_s}(f_1,\dots,f_s))=\det(C_\sigma)\mathrm{wr}_{\eta_1,\dots,\eta_s}(f_1,\dots,f_s).$$ From \cite[Chapter~II, \S1, Theorem~1]{Kol73}, there are $\theta_1,\dots,\theta_s\in\Theta$
with $\Lambda:=\mathrm{wr}_{\theta_1,\dots,\theta_s}(f_1,\dots,f_s)\neq0$.
Therefore, for any $\theta\in\Theta$ and $1\leq i\leq s$, the quotient
\begin{equation*}
	a_i(\theta):=
	\frac{\mathrm{wr}_{\theta_1,\dots,\theta_{i-1},\theta,\theta_{i+1},\dots,\theta_s}(f_1,\dots,f_s)}{\Lambda}
\end{equation*}
is fixed by each element of $\mathscr{G}(L|K)$. Thus, $a_i(\theta)\in L^{\mathscr{G}(L|K)}=K$. Now, by Cramer's
rule, $\theta f_j=\sum_{i=1}^s a_i(\theta)\,\theta_if_j$ for all $j$ and we obtain
\begin{equation}\label{finthering}V\subset \mathrm{Span}_K\{\theta f_j:\theta\in\Theta,\ j\in\{1,2,\dots,s\}\}\subseteq \mathrm{Span}_K\left\{\theta_i f_j: i,j\in \{1,2,\dots,s\}\right\}.\end{equation}

Let $Z:=(\theta_if_j)$. Since $\Lambda\neq 0$, we see that $Z\in \mathrm{GL}_s(L)$. From \eqref{G-actiononsolution}, for any $\sigma\in \mathscr G(L|K)$ and for any $\delta\in \Delta$,  we have $\sigma(Z)=ZC_\sigma$ and $\sigma(\delta(Z))=\delta(\sigma(Z))=(\delta Z) C_\sigma$. Therefore, $\sigma\left(\delta(Z)Z^{-1}\right)=\delta(Z)Z^{-1}$ for all $\sigma\in \mathscr G(L|K)$. Thus, $\delta_i(Z)=A_i Z$ for some $A_i\in \mathrm M_s(K)$ for $i=1,\dots, r$. Since entries of $Z$ are in the $\Delta$-field $L$, we must have $\delta_i\delta_j(Z)=\delta_j\delta_i(Z)$. From this, it follows that the system  $$\delta_i(Y)=A_i Y, \qquad i=1,\dots, r$$ is integrable. Since $C(L,\Delta)=C$, the ring $R:=K[Z,\Lambda^{-1}]$ is a Picard-Vessiot
ring over $K$ for this system. Clearly $R\subset L$. Finally, from the definition of $Z$ and from \eqref{finthering}, we obtain that $V\subset R$. \end{proof}

\begin{proposition}\label{rem-fd-span-PVring}
 Let $L$ be a $\Delta$-field extension of $K$ with $C(L,\Delta)=C$. If $W\subset L$ is a finite-dimensional $\Delta$-stable
$K$-subspace of $L$, then $W$ is contained in a Picard-Vessiot ring
over $K$.
\end{proposition}

\begin{proof}
Let $w_1,\dots,w_s$ be a $K$-basis of $W$ and put
$w:=(w_1,\dots,w_s)^{\mathsf t}$. Since $W$ is $\Delta$-stable, for each $i\in\{1,\dots,r\}$ we have
$\delta_i w=A_i w$ with $A_i\in\mathrm{M}_s(K)$, and since
$\delta_i\delta_jw=\delta_j\delta_iw$ for all $i,j\in\{1,\dots,r\}$, we obtain that $(\delta_iA_j-\delta_jA_i-A_iA_j+A_jA_i)w=0$ as well. Now we use the fact that $w_1,\dots, w_s$ are $K$-linearly independent, to conclude that $\delta_iA_j-\delta_jA_i-A_iA_j+A_jA_i=0$, that is, the system $\delta_i Y=A_i Y, i=1,\dots, r$ is integrable. Let $E$ be a Picard-Vessiot extension of $L$ for the system and  $Z\in\operatorname{GL}_s(E)$ be a fundamental solution matrix for the system. Let $R$ be the $K$-algebra generated by the entries of $Z$ and by $1/\mathrm{det} Z$. Then, since $C(E,\Delta)=C$, $R$ is a Picard-Vessiot ring over $K$ for the system. Observe that $\delta(Z^{-1}w)=0$ for all $\delta\in\Delta$ and therefore $w=Zx$ for some $x\in C^s$. Since the entries of $Z$ belong to $R$, we obtain that $w_1,\dots,w_s\in R$ and consequently,  $W\subset R$.
\end{proof}

\begin{proposition}\label{lemma-holo-liouv}
Let $M$ be a Liouvillian extension of $K$. If $f\in M$ lies in a
finite-dimensional $\Delta$-stable $K$-subspace of $K_\infty$, then $f$ lies
in a Liouvillian Picard-Vessiot extension of $K$.
\end{proposition}

\begin{proof}
By Proposition~\ref{rem-fd-span-PVring}, applied with $K_\infty$ in place of
$L$, there is a Picard-Vessiot extension $L$ of $K$ whose Picard-Vessiot
ring contains $f$. By Theorem~\ref{char-PV-ring},
\begin{equation*}
    V_{f}:=\operatorname{Span}_{C}\{\sigma(f):\sigma\in\mathscr{G}(L|K)\}
\end{equation*}
is a finite-dimensional $C$-vector space. Let $K\langle V_{f}\rangle$ denote
the $\Delta$-field generated over $K$ by $V_f$. Since $V_f$ is
$\mathscr{G}(L|K)$-stable, $K\langle V_{f}\rangle$ is a Picard-Vessiot
extension of $K$ by Proposition~\ref{remark-orbit-general}. Note that
$f\in K\langle V_{f}\rangle$. Thus, we may replace $L$ with
$K\langle V_{f}\rangle$. Let $\sigma_{1}(f),\dots,\sigma_{s}(f)$ be a
$C$-basis of $V_{f}$, with $\sigma_k\in\mathscr{G}(L|K)$. By
Theorem~\ref{lemma-aut-extend}(i), each $\sigma_k$ extends to a
$K$-$\Delta$-automorphism $\widetilde{\sigma}_k$ of $K_\infty$. Since
$\widetilde{\sigma}_k$ fixes $K$, each $\widetilde{\sigma}_k(M)$ is readily
seen to be a Liouvillian extension of $K$. Then, the compositum
$N:=\widetilde{\sigma}_{1}(M)\cdots\widetilde{\sigma}_{s}(M)\subseteq
K_\infty$ is also a Liouvillian extension of $K$ and it also contains $V_f$.
This implies $L\subseteq N$. Now from \cite[Theorem~1.43]{vdPS03}, it
follows that $L$ is a Liouvillian Picard-Vessiot extension of $K$.
\end{proof}

For the rest of this section,  $k$, $C$, $K$, $\Delta$ and $D$ are as defined in Section~\ref{intro}. As in Section \ref{sec-in-ring}, to avoid triviality, we shall also assume $C(K,D)=C$.

\subsection{Proof of Theorem \ref{main-thm-B}}
 Recall that a Liouvillian first integral of \eqref{system-fode}, by definition, lies in a Liouvillian extension of $K$, which in turn can be embedded in a tower of Liouvillian Picard-Vessiot
extensions over $K$. An induction on the length of this tower will be used in the proof of Theorem \ref{main-thm-B}. The base case follows from Corollary~\ref{thm-in-Liouv-PV}, and the
inductive step is presented in Lemmas~\ref{lemma-stepA} and~\ref{lemma-trade}. These two lemmas in turn rely on the following simple yet important observation.

\begin{proposition}\label{lemma-gradient}
	Let $E$ be a $\Delta$-field extension of $K$ and $e\in C(E,D)$ be a
	first integral of \eqref{system-fode}. Then
	$(\delta_0e,\dots,\delta_ne)^{\mathsf t}$ is a solution of the system
	\begin{equation}\label{equ-jac-D}
		DY=-(\delta_if_j)_{0\leq i,j\leq n}\,Y
	\end{equation}
	over the $D$-field $K$.
\end{proposition}

\begin{proof}
	Since $De=\sum_{j=0}^{n}f_j\delta_je=0$, applying $\delta_i$ and using
	the fact that the derivations in $\Delta$ commute with one another, we get
	\[
	0=\delta_i(De)=\sum_{j=0}^{n}\delta_i(f_j)\,\delta_je
	+\sum_{j=0}^{n}f_j\,\delta_j(\delta_ie)
	=\sum_{j=0}^{n}\delta_i(f_j)\,\delta_je+D(\delta_ie),
	\]
	that is, $D(\delta_ie)=-\sum_{j=0}^{n}\delta_i(f_j)\,\delta_je$ for
	$i=0,\dots,n$.
\end{proof}

\begin{lemma}\label{lemma-stepA}
Let $K\subseteq F\subset \mathcal E\subset K_\infty$ be $\Delta$-fields,
where $\mathcal E$ is a normal extension of $F$ such that
$C(\mathcal E, D)=C$. Let $e\in K_\infty$ be a first integral of the system
\eqref{system-fode} and let $e^*$ be either $1$ or $e$. Suppose that
$\delta_{i}(e)/e^*\in\mathcal E$ for all $0\leq i\leq n$. Then
\begin{enumerate}[(i)]
\item \label{pseudologderivatives} there are a Picard-Vessiot extension $L$
of $F$ with $L\subseteq \mathcal E$ and a $\Delta$-stable,
finite-dimensional $F$-subspace $V$ of $L$ such that
$\delta_i(e)/e^*\in V$ for each $0\leq i\leq n$;
\item \label{expintegral-constant} the system \eqref{system-fode} has a
first integral $\eta\in K_\infty$ such that $\delta_i(\eta)=\delta_i(e)/e^*$
for all $0\leq i\leq n$; in particular, $V+\mathrm{Span}_F\{\eta\}$ is a
$\Delta$-stable, finite-dimensional $F$-subspace of $K_\infty$.
\end{enumerate}
\end{lemma}

\begin{proof}
Let $g_{i}:=\delta_{i}e/e^*\in\mathcal E$. Since $e^*\in C(K_\infty, D)$, it
follows from Proposition~\ref{lemma-gradient}, applied with $K_\infty$ in
place of $E$, that $g:=(g_{0},\dots,g_{n})^{\mathsf t}$ is also a solution
of \eqref{equ-jac-D}. Since $C(\mathcal E, D)=C$, the set of all solutions
in $\mathcal E^{n+1}$ of \eqref{equ-jac-D} is a
$\mathscr G(\mathcal E|F)$-stable $C$-vector space of dimension at most
$n+1$. It follows that for each $0\leq i\leq n$,
$$W_i:=\operatorname{Span}_{C}\{\sigma(g_{i}):\sigma\in\mathscr{G}(\mathcal E|F)\}$$
is finite-dimensional. Let $W=W_0+\cdots+W_n$. Then $W$ is
$\mathscr G(\mathcal E|F)$-stable. Since $F_\infty=K_\infty$ and
$\mathcal E$ is a normal extension of $F$,
Theorem~\ref{lemma-aut-extend}(ii), applied with $F$ in place of $K$ and
$\mathcal E$ in place of $E$, gives
$\mathcal E^{\mathscr G(\mathcal E|F)}=F$. Hence
Proposition~\ref{remark-orbit-general}, applied with $F$, $\mathcal E$ and
$W$ in place of $K$, $L$ and $V$, shows that $\mathcal E$ contains a
Picard-Vessiot ring $R$ over $F$ such that $W\subset R$. Let
$w_1, \dots, w_l$ be a $C$-basis of $W$ and $L$ be the Picard-Vessiot
extension with Picard-Vessiot ring $R$. By Theorem~\ref{char-PV-ring},
applied to the Picard-Vessiot extension $L$ of $F$, each $w_{k}$ is
contained in a $\Delta$-stable, finite-dimensional $F$-subspace $V_{k}$ of
$L$. Consequently, their sum $V = V_1 + \dots + V_l$ is also a
$\Delta$-stable, finite-dimensional $F$-subspace of $L$. Since each $g_i$
lies in $W\subset V$, this proves (\ref{pseudologderivatives}).

Next, since $e^*$ is either $1$ or $e$, we have
$\delta_{j}g_{i}=\delta_{i}g_{j}$ for all $i,j$ and therefore the system
$$
\delta_{i}Y=g_{i},\qquad i=0,\dots,n
$$
is integrable. By Theorem~\ref{thm-K-inf}(ii), it then has a solution
$\eta\in K_\infty$. Note that
$$
D\eta=\sum_{i=0}^{n}f_{i}g_{i}=\frac{De}{e^*}=0.
$$
Since $e\notin C$, there is an $i$, $0\leq i\leq n$, such that
$g_{i}\neq0$. Then $\delta_i(\eta)=g_i\neq 0$ and therefore $\eta$ is a
first integral of \eqref{system-fode}. Since $\delta_i(\eta)=g_i\in V$ for
all $i$, the space $V+\mathrm{Span}_F\{\eta\}$ is a $\Delta$-stable,
finite-dimensional $F$-subspace of $K_\infty$. This proves
(\ref{expintegral-constant}).
\end{proof}

\begin{lemma}\label{lemma-trade}
Let $K\subseteq F_0\subseteq F_1\subseteq F_2\subset K_\infty$ be
$\Delta$-fields such that $C(F_0,D)=C$ and, for $i=1,2$, $F_i$ is a
Liouvillian Picard-Vessiot extension of $F_{i-1}$. Suppose that
\eqref{system-fode} has a first integral in $F_2$. Then \eqref{system-fode}
has a first integral in some Liouvillian Picard-Vessiot extension of $F_0$.
\end{lemma}

\begin{proof}
To avoid triviality, we may assume $C(F_1,D)=C$. Let
$\overline{F_1}\subset K_\infty$ be the algebraic closure of $F_1$. Then
$C(\overline{F_1},D)=C$ as well. Let $\mathscr G:=\mathscr G(F_2|F_1)$.
Applying Corollary~\ref{thm-in-Liouv-PV} with $F_1$ in place of $F$ and
$F_2$ in place of $E$, we obtain a first integral $e$ of the system
\eqref{system-fode} belonging to the Picard-Vessiot ring of the
Picard-Vessiot extension $F_2$ of $F_1$, such that for all
$\delta_i, \delta_j\in \Delta$ with $\delta_j(e)\neq 0$,
\begin{equation}\label{logderivinalgclosure}
\delta_i(\delta_j(e))/\delta_j(e)\in F^{\mathscr{G}^0}_2\subset
\overline{F_1}.
\end{equation}
Let $\mathcal E$ be the field generated over $\overline{F_1}$ by
$\{\delta_i(\sigma(e)): \sigma\in \mathscr G(K_\infty|F_0),
\delta_i\in \Delta\}$. Then $\mathcal E$ is a normal extension of $F_0$
containing each $\delta_i(e)$.

Suppose that $\mathcal E$ contains a first integral $\tilde{e}$ of the
system \eqref{system-fode}. Then one can find elements
$\sigma_1,\dots, \sigma_m\in \mathscr G(K_\infty|F_0)$ such that
$$\tilde{e}\in E:=\overline{F_1}\left(\delta_i(\sigma_j(e)):\delta_i\in
\Delta, \ j=1,2,\dots, m\right).$$
By \eqref{logderivinalgclosure}, and since each $\sigma_j$ maps
$\overline{F_1}$ onto itself, $E$ is generated over $\overline{F_1}$ by
elements whose logarithmic derivatives lie in $\overline{F_1}$. Hence $E$
is a Liouvillian Picard-Vessiot extension of $\overline{F_1}$, and
$\mathscr G(E|\overline{F_1})$ is connected and has a trivial unipotent
radical. Applying Corollary~\ref{thm-in-Liouv-PV} with $\overline{F_1}$ in
place of $F$ and the present $E$ in place of $E$, so that
$E^{\mathscr U}=E$ contains the first integral $\tilde e$ and
$E^{\mathscr{G}(E|\overline{F_1})^0}=\overline{F_1}$, we obtain a first
integral $e_1\in E$ such that $\delta_i(e_1)/e_1\in \overline{F_1}$ for
every $\delta_i\in \Delta$. Noting that $\overline{F_1}$ is a normal
extension of $F_0$ and that $C(\overline{F_1},D)=C$, we apply
Lemma~\ref{lemma-stepA} with $F_0$ in place of $F$, $\overline{F_1}$ in
place of $\mathcal E$ and $e_1$ in place of both $e$ and $e^*$, to obtain a first integral $\eta$
of \eqref{system-fode} and a $\Delta$-stable finite-dimensional
$F_0$-space $V\subset \overline{F_1}$ such that
$\delta_i(\eta)=\delta_i(e_1)/e_1\in V$ for all $i$. Since $V$ is finite
dimensional, we can choose a finite algebraic extension $\widetilde{F_1}$
of $F_1$ such that $V\subset \widetilde{F_1}$. Since $F_1$ is a Liouvillian
extension of $F_0$, so is $\widetilde{F_1}(\eta)$. 
Applying
Proposition~\ref{lemma-holo-liouv} with $F_0$ in place of $K$,
$\widetilde{F_1}(\eta)$ in place of $M$, $\eta$ in place of $f$ and noting that $\eta$ belongs to the 
finite-dimensional $\Delta$-stable $F_0$-space
$V+\mathrm{Span}_{F_0}\{\eta\}$, we conclude that $\eta$ belongs to a
Liouvillian Picard-Vessiot extension of $F_0$.

Finally, suppose that $C(\mathcal E, D)=C$. Applying
Lemma~\ref{lemma-stepA} with $F_0$ in place of $F$
and $e^*:=1$, we obtain a $\Delta$-stable finite-dimensional
$F_0$-subspace $V$ of $\mathcal E$ containing each $\delta_i(e)$. Then
$V+\mathrm{Span}_{F_0}\{e\}$ is a $\Delta$-stable finite-dimensional
$F_0$-subspace of $K_\infty$ containing $e$. Since $F_2$ is a Liouvillian
extension of $F_0$ containing $e$, applying
Proposition~\ref{lemma-holo-liouv} with $F_0$ in place of $K$, $F_2$ in
place of $M$ and $e$ in place of $f$ shows that $e$ belongs to a
Liouvillian Picard-Vessiot extension of $F_0$.
\end{proof}

\begin{proof}[Proof of Theorem~\ref{main-thm-B}, assertions (\ref{thmB1}) and (\ref{thmB2})]
Suppose that \eqref{system-fode} admits a Liouvillian first integral. If
\eqref{system-fode} admits a first integral in a Liouvillian Picard-Vessiot
extension $E$ of $K$, then Corollary~\ref{thm-in-Liouv-PV}, applied with $K$
in place of $F$, yields a first integral satisfying (\ref{thmB1}) and
(\ref{thmB2}) with $\widetilde K:=E^{\mathscr{G}(E|K)^0}$, which is a finite
algebraic extension of $K$. Hence, we only need to show that
\eqref{system-fode} admits a first integral in a Liouvillian Picard-Vessiot
extension of $K$. Let $e$ be a Liouvillian first integral of
\eqref{system-fode}. From Section~\ref{defn-LiouvExtn}, we know that $e$
lies in a $\Delta$-field $E_{\ell}$ admitting a tower
$$
K=E_{0}\subseteq E_{1}\subseteq\cdots\subseteq E_{\ell},
$$
where each $E_i$ is a Liouvillian Picard-Vessiot extension of $E_{i-1}$.
Since $e\in C(E_\ell,D)\setminus C$ and $C(E_0,D)=C$, we may choose the
least index $s\geq 1$ such that $C(E_{s},D)\neq C$ and discard the
extensions containing $E_{s}$, so that
\begin{equation}\label{equ-minimal-s}
	C(E_{j},D)=C\quad(0\leq j\leq s-1),\qquad C(E_{s},D)\neq C.
\end{equation}

If $s=1$, then we are done. Therefore, assume that $s\geq 2$ and apply
Lemma~\ref{lemma-trade} with $F_0=E_{s-2}$, $F_{1}=E_{s-1}$ and
$F_{2}=E_{s}$ to obtain a Liouvillian Picard-Vessiot extension $M$ of
$E_{s-2}$ in which \eqref{system-fode} has a first integral. Thus, we have
a tower of $\Delta$-fields
$$
K=E_{0}\subseteq E_{1}\subseteq\cdots\subseteq E_{s-2}\subseteq M,
$$
where for $1\leq i\leq s-2$, $E_i$ is a Liouvillian Picard-Vessiot
extension of $E_{i-1}$ and $M$ is a Liouvillian Picard-Vessiot extension
of $E_{s-2}$. This tower has length $s-1$ and $M$ contains a
first integral. Repeating this argument, one obtains a first integral of
\eqref{system-fode} in a Liouvillian Picard-Vessiot extension of $K$.
\end{proof}

\subsection{Singer's theorem revisited}\label{subsec-singer-descent}

For any autonomous planar system admitting a first integral $e$, we have $D(e)=f_1\delta_1(e)+f_2\delta_2(e)=0$. Since $\delta_2(e)\neq 0$, we have
$\delta_1e/\delta_2e=-f_2/f_1\in K$; that is, condition (\ref{thmB1})
of Theorem~\ref{main-thm-B} holds over $K$ itself. We show in the next proposition that, whenever (\ref{thmB1}) holds over $K$, there is a Liouvillian first integral satisfying both (\ref{thmB1}) and
(\ref{thmB2}) over $K$. Using the proposition, we will obtain Singer's theorem,
as stated in Section~\ref{intro}, as an easy consequence.

\begin{proposition}\label{prop-descent}
	Let $e$ and $\widetilde K$ be as in Theorem~\ref{main-thm-B}. If
	condition (\ref{thmB1}) holds over $K$, then \eqref{system-fode}
	admits a Liouvillian first integral satisfying (\ref{thmB1}) and
	(\ref{thmB2}) over $K$.
\end{proposition}

\begin{proof} Fix $0\leq j_0\leq n$ such that $\delta_{j_0}e\neq0$ and for notational convenience, let $$g_i:=\frac{\delta_i(e)}{\delta_{j_0}(e)}\in K\qquad \text{and}\qquad u_i:=\frac{\delta_i(\delta_{j_0}(e))}{\delta_{j_0}(e)}\in \widetilde{K}\quad \text{for } i=0,1,\dots,n.$$
   For $0\leq i,j\leq n$,  a calculation using the identity $\delta_j\delta_ie=\delta_i\delta_je$  gives
	\begin{align}\label{equ-desc-compat}
		\delta_jg_i-\delta_ig_j&=u_ig_j-u_jg_i,\\
       \label{commute-deltau} \delta_i(u_j)&=\delta_j(u_i).
	\end{align}
    Replace $\widetilde K$ by its Galois closure over $K$. Let $$\mathrm{Tr}_{\widetilde K|K}(x)=\sum_{\sigma\in \mathrm{Gal}(\widetilde{K}|K)}\sigma(x),\quad x\in \widetilde{K}.$$ Let $N=[\widetilde K: K]$ and 
	$$
	U_i :=\; \frac{1}{N}\,\mathrm{Tr}_{\widetilde K|K}(u_i)\;\in\;K,
	\qquad i=0,\dots,n.
	$$
Since every element of the Galois group $\mathrm{Gal}(\widetilde K|K)$ commutes with each $\delta_i$, from \eqref{commute-deltau}, we obtain that the system
    \begin{equation}\label{u-int-sys}
        \delta_iY=U_iY, \qquad 0\leq i\leq n
    \end{equation}
    is 
	integrable over $K.$ 
    
    Let $\rho$ be a nonzero solution of \eqref{u-int-sys} in $K_\infty$. Applying $\mathrm{Tr}_{\widetilde K|K}$ to \eqref{equ-desc-compat}, and using
	that $g_i$ and $\delta_jg_i-\delta_ig_j$ lie in $K$, gives
	$\delta_jg_i-\delta_ig_j=U_ig_j-U_jg_i$. Therefore $$\delta_j(\rho g_i)=\rho\left(U_jg_i+\delta_jg_i\right)=\rho\left(U_ig_j+\delta_ig_j\right)=\delta_i(\rho g_j).$$ Since $\rho g_i\in K(\rho)$, the system
	\begin{equation}\label{r-int-sys}
	    \delta_i Y=\rho g_i, \qquad i=0,\dots,n ,
	\end{equation}
	is integrable over $K(\rho)$. Let $\tilde e$ be a solution of \eqref{r-int-sys} in $K_\infty$. Then $K(\rho, \tilde{e})$ is a Liouvillian extension of $K$ and
	$$
D\tilde e=\sum_i f_i\,\rho g_i=\frac{\rho}{\delta_{j_0}e}\sum_i f_i\,\delta_ie
=\frac{\rho}{\delta_{j_0}e}\,De=0.
$$
	Since $\delta_{j_0}\tilde e=\rho\neq0$, we have $\tilde e\notin C$. Thus, $\tilde e$ is a Liouvillian first integral of the system \eqref{system-fode}. Finally, whenever
	$g_j\neq0$, 
	$\delta_i\tilde e/\delta_j\tilde e=g_i/g_j\in K$ and
	$\delta_i\delta_j\tilde e/\delta_j\tilde e=U_i+\delta_ig_j/g_j\in K$. This completes the proof.
\end{proof}

\begin{proof}[Proof of Theorem~\ref{main-thm-B}, final assertion]
Let $\Delta=\{\delta_a,\delta_b\}$, so that $D=f_a\delta_a+f_b\delta_b$ with
$f_a,f_b\neq0$. Let $e$ and $\widetilde K$ be as in
Theorem~\ref{main-thm-B} (\ref{thmB1}) and (\ref{thmB2}). Then
$f_a\delta_ae+f_b\delta_be=0$. If one of $\delta_ae$, $\delta_be$ were zero,
then so would be the other, and $e$ would be a $\Delta$-constant. Hence
$\delta_ae\neq0$, $\delta_be\neq0$ and
$\delta_ae/\delta_be=-f_b/f_a\in K$, that is, (\ref{thmB1}) holds over $K$.
Now, Proposition~\ref{prop-descent} gives a Liouvillian first integral
satisfying (\ref{thmB1}) and (\ref{thmB2}) over $K$. This completes the
proof of Theorem~\ref{main-thm-B}.
\end{proof}

\begin{proof}[Proof of Singer's theorem]
Suppose that the system has a Liouvillian first integral. By the final
assertion of Theorem~\ref{main-thm-B}, there is a Liouvillian first integral
$\tilde e$ satisfying (\ref{thmB1}) and (\ref{thmB2}) over $K$. Now the
computation following Theorem~\ref{main-thm-B} in Section~\ref{intro}
produces from $\tilde e$ an integrating factor
$\lambda=\exp\int(U\,dy_1+V\,dy_2)$ of $f_2\,dy_1-f_1\,dy_2$ with
$U,V\in K$ and $\delta_2U=\delta_1V$. The converse is immediate, as noted in
Section~\ref{intro}.
\end{proof}

\subsection{Results of Aziz et al. revisited}\label{subsec-comparison}

Assume that $k=C$ so that \eqref{system-fode} is now
an autonomous system and $D=\sum_{i=1}^{n}f_i\delta_i$ is a rational
vector field on $C^{n}$. In \cite{ACPW25} and \cite{CPW25}, Liouvillian
first integrals of such vector fields are studied in the language of
differential forms. To the $C$-derivation $D$ one associates the $(n-1)$-form
$$
\Omega \;:=\; \sum_{i=1}^{n}(-1)^{i-1}f_i\,
dy_1\wedge\cdots\wedge\widehat{dy_i}\wedge\cdots\wedge dy_n,
$$
where $dy_1\wedge\cdots\wedge\widehat{dy_i}\wedge\cdots\wedge dy_n$ denotes the wedge product $dy_1\wedge\cdots\wedge dy_{i-1}\wedge dy_{i+1}\wedge\cdots\wedge dy_n$. Since $d\varphi=\sum_i\delta_i\varphi\,dy_i$ and
$dy_i\wedge dy_1\wedge\cdots\wedge\widehat{dy_i}\wedge\cdots\wedge dy_n
=(-1)^{i-1}\,dy_1\wedge\cdots\wedge dy_n$, one has, for any $1$-form
$\omega=\sum_i a_i\,dy_i$ and any $\varphi$ in a $\Delta$-field extension of
$K$,
\begin{equation}\label{equ-wedge-D}
	\omega\wedge\Omega=\Big(\sum_i a_if_i\Big)\,dy_1\wedge\cdots\wedge dy_n,
	\qquad
	d\varphi\wedge\Omega=(D\varphi)\,dy_1\wedge\cdots\wedge dy_n .
\end{equation}
Following \cite[Definition~3]{ACPW25}, a nonconstant element $\varphi$
in a Liouvillian extension of $K$ is a \emph{Liouvillian first
integral} of the vector field when $d\varphi\wedge\Omega=0$, which by
\eqref{equ-wedge-D} is equivalent to $D\varphi=0$. In this language, \cite[Theorem~3]{ACPW25} states that if the vector
field has a Liouvillian first integral, then there are a finite
algebraic extension $\widetilde K$ of $K$ and $1$-forms $\omega\neq0$,
$\alpha$ over $\widetilde K$ with
\begin{equation}\label{equ-acpw-def}
	\omega\wedge\Omega=0, \qquad d\omega=\alpha\wedge\omega, \qquad d\alpha=0,
\end{equation}
and $\int\omega\exp(-\int\alpha)$ is then a Liouvillian first integral.
We show that this follows from Theorem~\ref{main-thm-B}, and that $\omega$
can be taken over $K$, in the sense of \cite{CPW25}, if and only if
condition (\ref{thmB1}) of Theorem~\ref{main-thm-B} holds over $K$.

Let $e$ and $\widetilde K$ be as in Theorem~\ref{main-thm-B}. Fix
$j_0$ with $\delta_{j_0}e\neq0$, and put $\ell:=\delta_{j_0}e$,
$r_i:=\delta_ie/\ell$ and $b_i:=\delta_i\ell/\ell$. Then,
$r_i,b_i\in\widetilde K$ by (\ref{thmB1}) and (\ref{thmB2}),
$r_{j_0}=1$, and
\begin{equation}\label{equ-e-ell}
	\delta_ie=r_i\ell,\qquad \delta_i\ell=b_i\ell\qquad(1\leq i\leq n).
\end{equation}
Since the derivations commute, applying $\delta_j$ to
$\delta_ie=r_i\ell$ and to $\delta_i\ell=b_i\ell$ and comparing with
the same expressions with $i$ and $j$ interchanged gives
\begin{equation}\label{equ-rb-compat}
	\delta_jr_i-\delta_ir_j=r_jb_i-r_ib_j,
	\qquad
	\delta_jb_i=\delta_ib_j
	\qquad(1\leq i,j\leq n).
\end{equation}

\begin{proposition}\label{prop-dictionary}
	Let
	\[
	\omega:=\frac{de}{\ell}=\sum_{i=1}^{n} r_i\,dy_i ,
	\qquad
	\alpha:=-\frac{d\ell}{\ell}=-\sum_{i=1}^{n} b_i\,dy_i .
	\]
	Then $\omega\neq0$ and $\alpha$ are $1$-forms over $\widetilde K$
	satisfying \eqref{equ-acpw-def}, $\exp(\int\alpha)=1/\ell$, and the
	first integral $\int\omega\exp(-\int\alpha)$ is $e$ itself.
\end{proposition}

\begin{proof}
	The coefficients of $\omega$ and $\alpha$ lie in $\widetilde K$, and since $r_{j_0}=1$ we have $\omega\neq0$. By \eqref{equ-wedge-D},
	$\ell\,\omega\wedge\Omega=de\wedge\Omega=(De)\,dy_1\wedge\cdots\wedge dy_n=0$.
	Next, since $\omega=\sum_j r_j\,dy_j$ and $\alpha=-\sum_i b_i\,dy_i$,
\begin{align*}
d\omega&=\sum_{i,j}\delta_ir_j\,dy_i\wedge dy_j
=\sum_{i<j}\big(\delta_ir_j-\delta_jr_i\big)\,dy_i\wedge dy_j,\\
\alpha\wedge\omega&=-\sum_{i,j}b_ir_j\,dy_i\wedge dy_j
=\sum_{i<j}\big(r_ib_j-r_jb_i\big)\,dy_i\wedge dy_j,\\
d\alpha&=-\sum_{i,j}\delta_ib_j\,dy_i\wedge dy_j
=-\sum_{i<j}\big(\delta_ib_j-\delta_jb_i\big)\,dy_i\wedge dy_j .
\end{align*}
Hence by \eqref{equ-rb-compat}, $d\omega=\alpha\wedge\omega$ and $d\alpha=0$. Finally $\alpha=-d\ell/\ell$ gives $\exp(\int\alpha)=1/\ell$, hence $\omega\exp(-\int\alpha)=\ell\omega=de$, and $e=\int\omega\exp(-\int\alpha)$.
\end{proof}

Following  \cite[Definition 2]{CPW25}, we call the system \eqref{system-fode} \emph{regular} if \eqref{equ-acpw-def} holds with $\omega \neq 0$ over $K$ and $\alpha$ over some Liouvillian extension of $K$; otherwise, it is called \emph{exceptional}.
 Exceptional systems exist for $n=3$ \cite[Section 5.2]{CPW25}. The next proposition identifies regularity with 
condition (\ref{thmB1}) of Theorem~\ref{main-thm-B} holding over $K$.

\begin{proposition}\label{prop-regular}
	The system \eqref{system-fode} is regular if and only if it admits a
	Liouvillian first integral $e$ with $\delta_ie/\delta_je\in K$ for all
	$i,j$ with $\delta_je\neq0$, that is, if and only if condition (\ref{thmB1}) of Theorem~\ref{main-thm-B}  holds over $K$. In that case $\alpha$ may also be taken over $K$.
\end{proposition}

\begin{proof}
    Let $e$ be such a first integral and choose $j_0$ with
	$\delta_{j_0}e\neq0$. Put $\lambda:=\delta_{j_0}e$, $\omega:=de/\lambda$
and $\alpha:=-d\lambda/\lambda$, a $1$-form over the Liouvillian extension
of $K$ containing $e$. Then $\omega\neq0$ has coefficients
	$\delta_ie/\lambda\in K$, and $\omega\wedge\Omega=0$ by
	\eqref{equ-wedge-D}. Clearly $d\alpha=0$, and applying $d$ to the equation $de=\lambda\omega$ gives
	$0=d\lambda\wedge\omega+\lambda\,d\omega$, that is,
	$d\omega=\alpha\wedge\omega$. Hence, the system is regular. If $e$ is
	chosen as in Proposition~\ref{prop-descent}, then
	$\delta_i\lambda/\lambda\in K$ by (\ref{thmB2}), so $\alpha$ is defined over $K$.

		Conversely, suppose \eqref{equ-acpw-def} holds with $\omega\neq0$
	over $K$, and write $\omega=\sum_i\omega_i\,dy_i$,
	$\alpha=\sum_i\alpha_i\,dy_i$. Since $d\alpha=0$, the system
	$\delta_iY=-\alpha_iY$, $1\leq i\leq n$, is integrable. Therefore, there is a
	nonzero $m$ in a Liouvillian extension of $K$ with $dm/m=-\alpha$.
	Then $d(m\omega)=m(-\alpha\wedge\omega+d\omega)=0$ and the system
	$\delta_iY=m\omega_i$, $1\leq i\leq n$, is integrable. It also has a
	solution $e$ in a Liouvillian extension of $K$. Now from \eqref{equ-wedge-D},
	$de\wedge\Omega=m\,\omega\wedge\Omega=0$ gives $De=0$. Since $de=m\omega\neq 0, e$ is nonconstant, and we have
	$\delta_ie/\delta_je=\omega_i/\omega_j\in K$.
\end{proof}

%%%%%%%%%%%%%%%%%%%%%%%%%%%%%%%%%%%%%%%%%%%%%%%%%%%%%%%%%%%%%%%%%%%%%%%%%%%%%
\section{Non-Liouvillian first integrals}\label{sec-non-liouv}
%%%%%%%%%%%%%%%%%%%%%%%%%%%%%%%%%%%%%%%%%%%%%%%%%%%%%%%%%%%%%%%%%%%%%%%%%%%%%

\subsection{The Liouvillian hypothesis in Theorem~\ref{main-thm-A} cannot be dropped}
\label{subsec-counter}
Let $k = \mathbb{C}$ and consider the autonomous planar system
\begin{equation}\label{example-counter-sys}
  \begin{cases}
    y_1' = 1,\\
    y_2' = -y_2^2 + y_1.
  \end{cases}
\end{equation}
Let $K = \mathbb{C}(y_1, y_2)$ be the rational function field, equipped with
the set of derivations $\Delta = \{\delta_1, \delta_2\}$, so that
$D = \delta_1 + (y_1 - y_2^2)\delta_2$. Consider the integrable system
\begin{equation}\label{equ-counter-pv}
  \delta_1 Y = \begin{pmatrix} 0 & 1 \\ y_1 & 0 \end{pmatrix} Y, \qquad
  \delta_2 Y = 0
\end{equation}
over $K$.\footnote{The equation
$\delta_1 Y = \left(\begin{smallmatrix} 0 & 1 \\ y_1 & 0
\end{smallmatrix}\right) Y$ is the matrix form of the classical Airy
equation $\delta_1^2y=y_1y$.}

Let $E$ be the Picard-Vessiot extension of $K$ for the system
\eqref{equ-counter-pv}, let $R \subset E$ be its Picard-Vessiot ring, and
let $Z\in \mathrm{GL}_2(E)$ be a fundamental solution matrix. Since
$\delta_2 Z=0$ and the second row of $Z$ is the $\delta_1$-derivative of the first, we have
$$Z = \begin{pmatrix}
    u_1 & u_2 \\ \delta_1(u_1) & \delta_1(u_2)
\end{pmatrix}, \qquad \delta_2(u_1)=\delta_2(u_2)=0,$$
 and
$E = K(u_1, u_2, \delta_1 u_1, \delta_1 u_2)$. Since the matrix in
\eqref{equ-counter-pv} has trace zero, $\det(Z)$ is a constant, which we
may assume to be $1$. Thus, the assignment $\sigma\mapsto C_\sigma$, defined
by $\sigma(Z)=ZC_\sigma$, embeds $\mathscr{G}(E|K)$ as a closed subgroup of
$\mathrm{SL}_2(\mathbb{C})$ (see Section~\ref{galoiscorrespondence}).

Let $M:=\mathbb{C}(y_1)(u_1, u_2, \delta_1 u_1, \delta_1 u_2)\subset E$. Note that $M$ is a $\delta_1$-field and that  $\delta_2$ is the zero derivation on $M$. Thus, $C(M,\delta_1)\subseteq C(E,\Delta)=\mathbb{C}$ and therefore
$M$ is the Picard-Vessiot extension of the $\delta_1$-field
$\mathbb{C}(y_1)$ for the scalar differential equation $\delta^2_1y=y_1y.$ It is well known that the differential Galois group of $M$ over
$\mathbb{C}(y_1)$ is $\mathrm{SL}_2(\mathbb{C})$ (see
\cite[Example 4.29]{Mag94}). Note that $E=M(y_2)$. Since $\delta_2y_2=1$ and that $\delta_2$ is the zero derivation on  $M$, $y_2$ must be transcendental over $M$. Now, every $\sigma\in\mathscr{G}(M|\mathbb C(y_1))$ naturally extends to a $K$-$\Delta$-automorphism of $E$; $\sigma(y_2)=y_2$ and $\sigma(Z)=ZC_\sigma$. Therefore the embedding
$\sigma\mapsto C_\sigma$ of $\mathscr{G}(E|K)$ into
$\mathrm{SL}_2(\mathbb{C})$ is surjective and
we obtain $\mathscr{G}(E|K)\cong\mathrm{SL}_2(\mathbb{C})$. In particular,
$\mathrm{tr.deg}(E|K)=\dim \mathrm{SL}_2(\mathbb{C})=3$, and therefore
$u_1$, $u_2$ and $\delta_1 u_1$ are $K$-algebraically independent.

We claim that the system~\eqref{example-counter-sys} admits a first integral
in the Picard-Vessiot extension $E$ of $K$, but none in its
Picard-Vessiot ring $R$.

\noindent\textbf{Existence of a first integral.} Let $e := p_2 / p_1$,
where for $i = 1, 2$,
$$p_i := u_i y_2 - \delta_1 u_i.$$
A direct calculation shows that $De = 0$ and
$\delta_2(e) = 1 / p_1^2\neq 0$. Thus, $e$ is a first integral of the
system~\eqref{example-counter-sys}.

\noindent\textbf{The first integral $e$ does not belong to $R$.} For
any $\sigma \in \mathscr{G}(E|K)$ and $j = 1, 2$, we have
$\sigma(u_j) = c_{1j\sigma} u_1 + c_{2j\sigma} u_2$, where
$(c_{ij\sigma}) \in \mathrm{SL}_2(\mathbb{C})$. It follows that
$\sigma(p_j) = c_{1j\sigma} p_1 + c_{2j\sigma} p_2$. This yields
$$\sigma(e) = \frac{c_{12\sigma} + c_{22\sigma} e}{c_{11\sigma} + c_{21\sigma} e}
\qquad \text{for all } \sigma \in \mathscr{G}(E|K).$$
Let $W = \mathrm{Span}_{\mathbb{C}} \{ \sigma(e) : \sigma \in
\mathscr{G}(E|K) \}$. For each $c \in \mathbb{C}$, consider the
$K$-$\Delta$-automorphism of $E$ induced by the matrix
$(c_{ij\sigma}) \in \mathrm{SL}_2(\mathbb{C})$ with $c_{11\sigma} = c$,
$c_{12\sigma} = -1$, $c_{21\sigma} = 1$, and $c_{22\sigma} = 0$. This
$K$-$\Delta$-automorphism maps $e$ to $\frac{-1}{c + e}$, which implies that
$$\mathrm{Span}_{\mathbb{C}} \left\{ \frac{1}{e + c} : c \in \mathbb{C}
\right\} \subset W.$$
Since $e$ is transcendental over $\mathbb{C}$, the elements $1/(e+c)$,
$c\in\mathbb{C}$, are linearly independent over $\mathbb{C}$. Consequently,
$W$ is infinite-dimensional over $\mathbb{C}$ and, by
Theorem~\ref{char-PV-ring}, $e \notin R$.

\noindent\textbf{Non-existence of first integrals in the Picard-Vessiot ring
$R$.} Consider the $\Delta$-field $L:=K(u_1, \delta_1 u_1)$, which is
the fixed field of the subgroup of all upper triangular matrices in
$\mathrm{SL}_2(\mathbb{C})$ having the value $1$ on the diagonal. Observe
that $D(u_1) = \delta_1 u_1$ and
$D(\delta_1 u_1) = \delta_1^2 u_1 = y_1 u_1$.

We first claim that $L$ contains no first integrals of the system
\eqref{example-counter-sys}, that is, $C(L,D)=\mathbb{C}$. Note that $D$
and $\delta_1$ agree on $M$, so that $M$ is a $D$-subfield of $E$
with $C(M,D)=\mathbb{C}$. Put $w:=\delta_1(u_2)/u_2\in M$. Then $w$ is a
solution of the Riccati differential equation $\delta_1(w) = -w^2 + y_1$.
Since $\det(Z)=1$, we have $u_2\,(u_1w-\delta_1u_1)=1$, and therefore
$\mathbb{C}(y_1)(u_1,\delta_1u_1,w)$ contains $u_2$ and
$\delta_1u_2=wu_2$; that is, it equals $M$. As
$\mathrm{tr.deg}(M|\mathbb{C}(y_1))=3$, the elements $u_1$,
$\delta_1 u_1$ and $w$ are algebraically independent over
$\mathbb{C}(y_1)$. Hence there is a $\mathbb{C}(y_1)$-embedding of fields
$\varphi\colon L\to M$ with
$$\varphi(y_2)=w,\qquad \varphi(u_1)=u_1,\qquad
\varphi(\delta_1u_1)=\delta_1u_1.$$
Since $D(y_2)=-y_2^2+y_1$ and $D(w)=-w^2+y_1$, and since
$D(u_1) = \delta_1 u_1$ and $D(\delta_1 u_1) = y_1 u_1$, the map $\varphi$
is an embedding of $D$-fields. Since
$C(M, D)=\mathbb{C}$, we obtain that $C(L,D)=\mathbb{C}$. This proves the
claim.

We shall now show that the Picard-Vessiot ring $R$ contains no first
integrals of the system \eqref{example-counter-sys}. Assume for a
contradiction that $e_1 \in R$ is a first integral. Then both $e$
and $e_1$ belong to $E \setminus L$. We have
$$\delta_2(e) = \frac{1}{p_1^2} \in L, \qquad
\delta_1(e) = \frac{y_2^2 - y_1}{p_1^2} \in L,$$
and $E$ is a Picard-Vessiot extension of $L$ with a differential
Galois group isomorphic to the additive group $\mathbb{G}_a$. Since
$e\notin L$ and $\mathbb{G}_a$ has no proper nontrivial algebraic subgroups,
it follows that $E = L(e)$. Moreover, $L[e]$ is the Picard-Vessiot
ring of $E$ over $L$ and therefore $R\subset L[e]$. Write
$$e_1 = a_l e^l + a_{l-1} e^{l-1} + \dots + a_0, \quad
\text{with } a_i\in L,\ a_l \neq 0 \text{ and } l \geq 1.$$
Since $De=0$, we have $0=D(e_1)=\sum_{i=0}^{l}D(a_i)e^i$, and since $e$ is
transcendental over $L$, $D(a_i)=0$ for all $i$. Because
$C(L,D)=\mathbb{C}$, we get $e_1=q(e)$ for a nonconstant polynomial
$q\in\mathbb{C}[X]$. For each $c\in\mathbb{C}$, let $\sigma_c\in\mathscr{G}(E|K)$ be the
automorphism considered above, such that $\sigma_c(e)=-1/(e+c)$. Then
$\sigma_c(e_1)=q\big(-1/(e+c)\big)$ is a rational function of $e$ whose
only pole is at $e=-c$. This implies that
$\mathrm{Span}_{\mathbb{C}}\{\sigma(e_1):\sigma\in\mathscr{G}(E|K)\}$ is not
finite-dimensional. By Theorem~\ref{char-PV-ring}, $e_1\notin R$, a
contradiction. Consequently, the system possesses no first integrals in the
Picard-Vessiot ring $R$.

%%%%%%%%%%%%%%%%%%%%%%%%%%%%%%%%%%%%%%%%%%%%%%%%%%%%%%%%%%%%%%%%%%%%%%%%%%%%%%%%%%%%%%%%%%%%%%%%%%%%%%%%%%%%%%%%%%%%%%%%%%%%%%%%%%%%%%%%%%%%%%%%%%%%%%%%%%%%%%%%%%%%%%%%%%%%%%%%%%%%%%
\subsection{A first integral in a Picard-Vessiot ring need not be Liouvillian}
\label{subsec-airy}

Theorem~\ref{main-thm-A} produces a first integral in the Picard-Vessiot
ring from a first integral in a Liouvillian Picard-Vessiot extension. We
shall again use the Airy differential equation to show that, without the
Liouvillian hypothesis, a first integral lying in a Picard-Vessiot ring need not be Liouvillian.

Let $k = \mathbb{C}$ and consider the system
\begin{equation}\label{example-in-ring}
  \begin{cases}
    y_1' = 1,\\
    y_2' = y_3,\\
    y_3' = y_1 y_2 .
  \end{cases}
\end{equation}
Here, $K = \mathbb{C}(y_1, y_2, y_3)$, with
$\Delta = \{\delta_1, \delta_2, \delta_3\}$, and
$D = \delta_1 + y_3 \delta_2 + y_1 y_2 \delta_3$. Consider the integrable
system
\begin{equation}\label{equ-airy-pv}
  \delta_1 Y = \begin{pmatrix} 0 & 1 \\ y_1 & 0 \end{pmatrix} Y, \qquad
  \delta_2 Y = \delta_3 Y = 0
\end{equation}
over $K$. As in Section \ref{subsec-counter}, let
$E = K(u_1, u_2, \delta_1 u_1, \delta_1 u_2)$ be the Picard-Vessiot
extension of $K$ for the system~\eqref{equ-airy-pv}, with Picard-Vessiot
ring $R=K[u_1, u_2, \delta_1 u_1, \delta_1 u_2]$, where
$u_1,u_2\in R$ are such that
$$Z := \begin{pmatrix} u_1 & u_2 \\ \delta_1 u_1 & \delta_1 u_2 \end{pmatrix}
\in \mathrm{SL}_2(R)$$
is a fundamental solution matrix of the system. As before,
$\mathscr{G}(E|K)\cong\mathrm{SL}_2(\mathbb{C})$. For the elements
$$e_1 := y_2 \delta_1 u_2 - y_3 u_2, \qquad e_2 := y_3 u_1 - y_2 \delta_1 u_1,$$
we have\footnote{These are the constants of integration:
$(y_2, y_3)^{\mathsf t} = Z (c_1, c_2)^{\mathsf t}$ exactly when
$(c_1, c_2) = (e_1, e_2)$.}
$(e_1, e_2)^{\mathsf t} = Z^{-1} (y_2, y_3)^{\mathsf t}$. Since
$\delta_2 u_i = \delta_3 u_i = 0$, it follows that $D u_i = \delta_1 u_i$
and $D(\delta_1 u_i) = y_1 u_i$. Given that $D y_2 = y_3$ and
$D y_3 = y_1 y_2$, direct calculation shows that
\begin{align*}
De_1 &= y_3 \delta_1 u_2 + y_2 y_1 u_2 - y_1 y_2 u_2 - y_3 \delta_1 u_2 = 0,\\
De_2 &= y_1 y_2 u_1 + y_3 \delta_1 u_1 - y_3 \delta_1 u_1 - y_2 y_1 u_1 = 0.
\end{align*}
Moreover, $\delta_2 e_1 = \delta_1 u_2 \neq 0$ and
$\delta_2 e_2 = -\delta_1 u_1 \neq 0$, which ensures
$e_1, e_2 \notin \mathbb{C}$. Thus, $e_1$ and $e_2$ are first integrals
of~\eqref{example-in-ring} inside the Picard-Vessiot ring $R$.

We shall now show that no element of $R\setminus K$ lies in a
Liouvillian extension of $K$; in particular, $e_1$ and $e_2$ are not
Liouvillian over $K$. Let $e\in R\setminus K$ and assume for a
contradiction that $e$ lies in a Liouvillian extension of $K$. Since
$e \in R$, it is contained in a $\Delta$-stable, finite-dimensional
$K$-subspace of $K_\infty$ by Theorem~\ref{char-PV-ring}. By
Proposition~\ref{lemma-holo-liouv}, $e$ belongs to a Liouvillian
Picard-Vessiot extension $F$ of $K$.

By \cite[Corollary 2, p.~402]{Kol73}, the field $E \cap F$ is
a Picard-Vessiot extension of $K$. Its Galois group is a quotient of
$\mathscr{G}(E|K) \cong \mathrm{SL}_2(\mathbb{C})$, hence connected,
and a quotient of $\mathscr{G}(F|K)$, hence solvable by
Theorem~\ref{Liou_PV}. The proper normal algebraic subgroups of
$\mathrm{SL}_2(\mathbb{C})$ are $\{I\}$ and $\{\pm I\}$, and neither
$\mathrm{SL}_2(\mathbb{C})$ nor $\mathrm{PSL}_2(\mathbb{C})$ is solvable.
Therefore, this Galois group is trivial, which yields
$e\in E \cap F=K$, a contradiction.

Thus, the system~\eqref{example-in-ring} admits first integrals in the
Picard-Vessiot ring $R$, none of which is Liouvillian over $K$.
Together with the example in Section~\ref{subsec-counter}, this shows that
the Liouvillian hypothesis in Theorem~\ref{main-thm-A} cannot be dropped,
and that a first integral in a Picard-Vessiot ring need not be Liouvillian.

\bibliographystyle{amsplain}
	\bibliography{KSS}
	
\end{document}